\documentclass[a4paper,reqno,11pt]{amsart}
\usepackage{amsmath,amsfonts,amssymb,amsthm}
\usepackage{bbm}
\usepackage[dvipsnames]{xcolor}
\usepackage[utf8]{inputenc}
\usepackage[T1]{fontenc}
\usepackage[margin = 1in]{geometry}
\usepackage{mathtools}
\usepackage{bbm}
\usepackage{enumitem}
\usepackage{textcomp}
\usepackage{graphicx}
\usepackage{float}
\usepackage{subcaption}
\usepackage{xcolor}
\usepackage{enumitem}
\usepackage[export]{adjustbox}
\usepackage{verbatim}

\usepackage{urwchancal}
\DeclareFontFamily{OT1}{pzc}{}
\DeclareFontShape{OT1}{pzc}{m}{it}%
              {<-> s * [1.15] pzcmi7t}{}
\DeclareMathAlphabet{\mathpzc}{OT1}{pzc}{m}{it}

\usepackage{tikz}

\usepackage[toc,page]{appendix}
\usepackage{cite}

\usepackage[hyphens]{url}
\usepackage[hidelinks, colorlinks, allcolors = black]{hyperref}
\usepackage{hyperref}

\makeatletter
\def\user@resume{resume}
\def\user@intermezzo{intermezzo}
\newcounter{previousequation}
\newcounter{lastsubequation}
\newcounter{savedparentequation}
\renewenvironment{subequations}[1][]{%
      \def\user@decides{#1}%
      \setcounter{previousequation}{\value{equation}}%
      \ifx\user@decides\user@resume 
           \setcounter{equation}{\value{savedparentequation}}%
      \else  
      \ifx\user@decides\user@intermezzo
           \refstepcounter{equation}%
      \else
           \setcounter{lastsubequation}{0}%
           \refstepcounter{equation}%
      \fi\fi
      \protected@edef\theHparentequation{%
          \@ifundefined {theHequation}\theequation \theHequation}%
      \protected@edef\theparentequation{\theequation}%
      \setcounter{parentequation}{\value{equation}}%
      \ifx\user@decides\user@resume 
           \setcounter{equation}{\value{lastsubequation}}%
         \else
           \setcounter{equation}{0}%
      \fi
      \def\theequation  {\theparentequation  \alph{equation}}%
      \def\theHequation {\theHparentequation \alph{equation}}%
      \ignorespaces
}{%
  \ifx\user@decides\user@resume
       \setcounter{lastsubequation}{\value{equation}}%
       \setcounter{equation}{\value{previousequation}}%
  \else
  \ifx\user@decides\user@intermezzo
       \setcounter{equation}{\value{parentequation}}%
  \else
       \setcounter{lastsubequation}{\value{equation}}%
       \setcounter{savedparentequation}{\value{parentequation}}%
       \setcounter{equation}{\value{parentequation}}%
  \fi\fi
  \ignorespacesafterend
}
\makeatother

\makeatletter
\newcommand{\mylabel}[2]{#2\def\@currentlabel{#2}\label{#1}}
\makeatother

\makeatletter
\newcommand{\proofstep}[1]{%
  \par
  \addvspace{\medskipamount}
  \textit{#1\@addpunct{.}}\enspace\ignorespaces
}
\makeatother

\renewcommand{\d}{\ensuremath{\mathrm{d}}}

\newcommand{\N}{\ensuremath{\mathbb{N}}}

\newcommand{\R}{\ensuremath{\mathbb{R}}}

\def\R{\mathbb{R}}

\def\N{\mathbb{N}}

\newcommand{\II}{\ensuremath{\mathcal I}}
\newcommand{\JJ}{\ensuremath{\mathcal J}}

\newcommand{\TT}{\ensuremath{\mathcal T}}

\newcommand{\rbar}[2]{\ensuremath{\left.#1\right|_{#2}}}

\newtheorem{theorem}{Theorem}[section]

\newtheorem{proposition}[theorem]{Proposition}

\newtheorem{lemma}[theorem]{Lemma}
\newtheorem{corollary}[theorem]{Corollary}

\numberwithin{equation}{section}
\newcommand{\cG}{{\mathcal G}}  
\newcommand{\cS}{{\mathcal S}}  
\newcommand{\cT}{{\mathcal T}}  

\newcommand{\Dualpair}[2]{\ensuremath{\left\langle \kern-0.5ex \left\langle #1,#2 \right\rangle \kern-0.5ex \right\rangle}}

\newcommand{\qvar}[2]{\ensuremath{\left\langle \kern-0.5ex \left\langle #1 \right\rangle \kern-0.5ex \right\rangle_{#2}}}

\newcommand{\eps}{\ensuremath{\varepsilon}}
\newcommand{\verti}[1]{\ensuremath{\left\lvert #1 \right\rvert}}
\newcommand{\vertii}[1]{\ensuremath{\left\lVert #1 \right\rVert}}
\newcommand{\vertiii}[1]{{\vertii{\kern-0.25ex\vertii{\kern-0.25ex\vertii{ #1
    }\kern-0.25ex}\kern-0.25ex}}}
\renewcommand{\d}{\ensuremath{{\rm d}}}

\newcommand{\order}{\ensuremath{\mathcal O}}
\newcommand{\landau}{\ensuremath{\mathpzc o}}

\begin{document}
%
\title[The Cox-Voinov law in the partial wetting regime]{The Cox-Voinov law for traveling waves in the partial wetting regime}
\thanks{The first author is grateful to Lorenzo Giacomelli and Felix Otto for discussions on work preceding this paper. The second author thanks Floris Roodenburg for discussions. This work is partially based on the second author's bachelor thesis in applied mathematics prepared under the advice of the first author at Delft University of Technology.}
\date{\today}
\author{Manuel V. Gnann}
\address[Manuel~V.~Gnann]{Delft Institute of Applied Mathematics, Faculty of Electrical Engineering, Mathematics and Computer Science, Delft University of Technology, Mekelweg 4, 2628 CD Delft, Netherlands}
\email{M.V.Gnann@tudelft.nl}
\author{Anouk C. Wisse}
\address[Anouk~C.~Wisse]{Delft Institute of Applied Mathematics, Faculty of Electrical Engineering, Mathematics and Computer Science, Delft University of Technology, Mekelweg 4, 2628 CD Delft, Netherlands}
\email{A.C.Wisse@student.tudelft.nl}
\begin{abstract}
We consider the thin-film equation $\partial_t h + \partial_y \left(m(h) \partial_y^3 h\right) = 0$ in $\{h > 0\}$ with partial-wetting boundary conditions and inhomogeneous mobility of the form $m(h) = h^3+\lambda^{3-n}h^n$, where $h \ge 0$ is the film height, $\lambda > 0$ is the slip length, $y > 0$ denotes the lateral variable, and $n \in (0,3)$ is the mobility exponent parameterizing the nonlinear slip condition. The partial-wetting regime implies the boundary condition $\partial_y h = \mathrm{const.} > 0$ at the triple junction $\partial\{h > 0\}$ (nonzero microscopic contact angle). Existence and uniqueness of traveling-wave solutions to this problem under the constraint $\partial_y^2 h \to 0$ as $h \to \infty$ have been proved in previous work by Chiricotto and Giacomelli in [Commun. Appl. Ind. Math., 2(2):e--388, 16, 2011]. We are interested in the asymptotics as $h \downarrow 0$ and $h \to \infty$. By reformulating the problem as $h \downarrow 0$ as a dynamical system for the difference between the solution and the microscopic contact angle, values for $n$ are found for which linear as well as nonlinear resonances occur. These resonances lead to a different asymptotic behavior of the solution as $h\downarrow0$ depending on $n$.

\medskip

Together with the asymptotics as $h\to\infty$ characterizing the Cox-Voinov law for the velocity-dependent macroscopic contact angle as found by Giacomelli, the first author of this work, and Otto in [Nonlinearity, 29(9):2497--2536, 2016], the rigorous asymptotics of traveling-wave solutions to the thin-film equation in partial wetting can be characterized. Furthermore, our approach enables us to analyze the relation between the microscopic and macroscopic contact angle. It is found that the Cox-Voinov law for the macroscopic contact angle depends continuously differentiably on the microscopic contact angle.
\end{abstract}
\maketitle
%

\section{Introduction}
\subsection{The thin-film equation formulated as a classical free-boundary problem}
The following thin-film equation with boundary conditions in a moving domain $(Y,\infty)$ is studied:
\begin{subequations}\label{TFE}
\begin{align}
    \partial_t h + \partial_y \left((h^3+\lambda^{3-n}h^n) \partial_y^3 h\right) &= 0 && \text{for }t>0\text{ and } y>Y, \label{TFE_PDE}\\
    h &=  0 && \text{for } t>0 \text{ and } y=Y, \label{BC:1}\\
    \partial_y h &= k>0 && \text{for } t>0 \text{ and } y=Y, \label{BC:2}\\
    \lim_{y \downarrow Y} (h^2+\lambda^{3-n} h^{n-1}) \partial_y^3 h &= \tfrac{\d Y}{\d t} &&\text{for } t>0. \label{BC:3}
\end{align}
\end{subequations}
Here, $h = h(t,y)$ denotes the height of a liquid thin film on a flat surface at time $t>0$ and base point $y \in (Y,\infty)$, where $Y$ is a function of time $t \ge 0$, which is visualized in Figure~\ref{fig:thin_film}.
\begin{figure}[htb]
    \centering
\begin{tikzpicture}[x=0.75pt,y=0.75pt,yscale=-1,xscale=1]
\draw   [color={rgb, 255:red, 74; green, 144; blue, 226 }] (97.05,154.63) .. controls (189.5,145.33) and (280.5,67.33) .. (319.5,64.33) ;
\draw  [draw opacity=0][fill={rgb, 255:red, 227; green, 227; blue, 227 }  ,fill opacity=1 ] (77.05,154.63) -- (301.5,154.63) -- (301.5,194.63) -- (77.05,194.63) -- cycle ;
\draw  (50,154.63) -- (320.5,154.63)(77.05,16.33) -- (77.05,170) (313.5,149.63) -- (320.5,154.63) -- (313.5,159.63) (72.05,23.33) -- (77.05,16.33) -- (82.05,23.33)  ;

\draw (169,166) node [anchor=north west][inner sep=0.75pt]   [align=left] {solid};
\draw (250,115) node [anchor=north west][inner sep=0.75pt]   [align=left] {liquid};
\draw (131,61) node [anchor=north west][inner sep=0.75pt]   [align=left] {gas};
\draw (300,156.4) node [anchor=north west][inner sep=0.75pt]    {$y$};
\draw (63,27.4) node [anchor=north west][inner sep=0.75pt]    {$z$};
\draw (286,54.4) node [anchor=north west][inner sep=0.75pt]    {$h$};
\end{tikzpicture}
\caption{Example of a thin film as described by \eqref{TFE}}
    \label{fig:thin_film}
\end{figure}
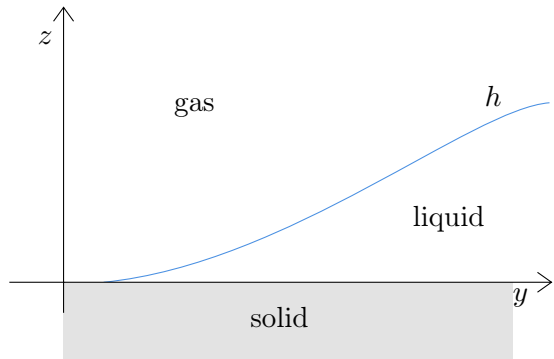
For simplicity we assume translation invariance in the third physical direction (perpendicular to the $(y,z)$-plane). Equation~\eqref{TFE_PDE} is a lubrication model, which means that it describes the flow of the fluid of a thin and viscous film in which the dynamics in the vertical direction $z$ are averaged out. It has the form of a continuity equation
\begin{equation*}
    \partial_t h + \partial_y \left(h u\right)=0,
\end{equation*}
where $h$ is the film height and $u$ is the velocity of the fluid in the horizontal direction $y$ which is averaged in the vertical direction $z$. In the case of equation \eqref{TFE_PDE}, the velocity of the flow $u$ is given by $u = (h^2+\lambda^{3-n}h^{n-1}) \partial_y^3 h$. The equation can be derived from the Navier-Stokes free-boundary problem, which has been done in detail for instance in \cite[Chapter~2, Section~B]{ODB}.

\medskip

The exponent $n$ is called the mobility exponent and we consider $n\in(0,3)$. This is because on one hand, if $n\leq0$ the speed of propagation of the film is infinite. On the other hand, in case of $n\geq3$ or $\lambda = 0$ (vanishing slip length), the boundary of the film does not move \cite{HuhScriven1971,DussanDavis1974}. Note that the regime $n\in(0,1)$ is physically not justified as well, as the film height $h$ can in certain situations become negative (see for instance \cite{BW}). Hence, our results for $n \in (0,1)$ should be considered as purely motivated from the mathematical perspective while the parameter regime $n \in [1,3)$ is of mathematical as well as physical interest. In particular, this interval contains the physically relevant values $n=1$ (free slip in the Hele-Shaw cell, see e.g.~\cite{GiacomelliOtto2003,KnuepferMasmoudi2013,KnuepferMasmoudi2015}, or the Greenspan slip condition \cite{Greenspan1978}) and $n=2$ (linear Navier slip, see e.g.~\cite{Navier1823,BEIMR,ODB,Knuepfer2011}).

\medskip

The film covers the interval $(Y,\infty)$ and has a free boundary at $y=Y$ called contact line or triple junction since it parametrizes the in our case straight but time-dependent line where liquid, gas, and solid meet. The trivial constraint \eqref{BC:1} entails that the height of the thin film at the triple junction is zero. Condition~\eqref{BC:2} implies that the contact angle between the solid and the film at the contact line is equal to $\theta = \arctan k$, where $k>0$ (partial-wetting regime). Since in lubrication approximation $k$ is necessarily small, we simply call $k$ the (microscopic) contact angle. The kinematic condition \eqref{BC:3} implies that, on approaching the contact line, the vertically averaged horizontal velocity $u$ is the same as the free boundary's velocity $\frac{\d Y}{\d t}$.

\subsection{Microscopic versus macroscopic contact angle}
The capillary forces acting at the triple junction are depicted in Figure~\ref{fig:young's_equation}.
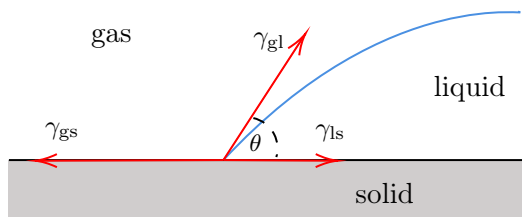
\begin{figure}[htp]
    \centering

\tikzset{every picture/.style={line width=0.75pt}}

\begin{tikzpicture}[x=0.75pt,y=0.75pt,yscale=-1,xscale=1]

\draw  [draw opacity=0][fill={rgb, 255:red, 208; green, 205; blue, 205 }  ,fill opacity=1 ] (69.5,111) -- (325.5,111) -- (325.5,143) -- (69.5,143) -- cycle ;
\draw [color={rgb, 255:red, 74; green, 144; blue, 226 }  ,draw opacity=1 ]   (176.5,111) .. controls (259.5,21) and (333.5,39.33) .. (323.5,36.33) ;
\draw  [draw opacity=0][dash pattern={on 4.5pt off 4.5pt}] (191.46,89.87) .. controls (195.74,90.82) and (199.63,93.69) .. (201.83,98.02) .. controls (203.72,101.75) and (203.99,105.85) .. (202.91,109.48) -- (188.86,104.6) -- cycle ; \draw  [dash pattern={on 4.5pt off 4.5pt}] (191.46,89.87) .. controls (195.74,90.82) and (199.63,93.69) .. (201.83,98.02) .. controls (203.72,101.75) and (203.99,105.85) .. (202.91,109.48) ;
\draw    (69.5,111) -- (325.5,111) ;
\draw [color={rgb, 255:red, 255; green, 0; blue, 0 }  ,draw opacity=1 ][fill={rgb, 255:red, 255; green, 0; blue, 0 }  ,fill opacity=1 ]   (176.5,111) -- (216.42,49.01) ;
\draw [shift={(217.5,47.33)}, rotate = 482.78] [color={rgb, 255:red, 255; green, 0; blue, 0 }  ,draw opacity=1 ][line width=0.75]    (10.93,-3.29) .. controls (6.95,-1.4) and (3.31,-0.3) .. (0,0) .. controls (3.31,0.3) and (6.95,1.4) .. (10.93,3.29)   ;
\draw [color={rgb, 255:red, 255; green, 0; blue, 0 }  ,draw opacity=1 ]   (176.5,111) -- (84.5,111.33) ;
\draw [shift={(82.5,111.33)}, rotate = 359.8] [color={rgb, 255:red, 255; green, 0; blue, 0 }  ,draw opacity=1 ][line width=0.75]    (10.93,-3.29) .. controls (6.95,-1.4) and (3.31,-0.3) .. (0,0) .. controls (3.31,0.3) and (6.95,1.4) .. (10.93,3.29)   ;
\draw [color={rgb, 255:red, 255; green, 0; blue, 0 }  ,draw opacity=1 ]   (176.5,111) -- (231.5,111.32) ;
\draw [shift={(233.5,111.33)}, rotate = 180.34] [color={rgb, 255:red, 255; green, 0; blue, 0 }  ,draw opacity=1 ][line width=0.75]    (10.93,-3.29) .. controls (6.95,-1.4) and (3.31,-0.3) .. (0,0) .. controls (3.31,0.3) and (6.95,1.4) .. (10.93,3.29)   ;

\draw (241,121) node [anchor=north west][inner sep=0.75pt]   [align=left] {solid};
\draw (280,67) node [anchor=north west][inner sep=0.75pt]   [align=left] {liquid};
\draw (109,44) node [anchor=north west][inner sep=0.75pt]   [align=left] {gas};
\draw (188,96.4) node [anchor=north west][inner sep=0.75pt]  [font=\footnotesize]  {${\textstyle \theta }$};
\draw (86,90.4) node [anchor=north west][inner sep=0.75pt]  [font=\small]  {$\gamma _\mathrm{gs}$};
\draw (190,45.4) node [anchor=north west][inner sep=0.75pt]  [font=\small]  {$\gamma _\mathrm{gl}$};
\draw (221,91.4) node [anchor=north west][inner sep=0.75pt]  [font=\small]  {$\gamma _\mathrm{ls}$};

\end{tikzpicture}

\caption{Surface tensions acting on a liquid at the triple junction.}
    \label{fig:young's_equation}
\end{figure}
Young's law (cf.~\cite{BEIMR})
\begin{equation}\label{young}
    \gamma_\mathrm{gs}=\gamma_\mathrm{ls}+\cos(\theta)\gamma_\mathrm{gl}
\end{equation}
gives the relation between the microscopic contact angle $\theta$ and the surface tensions $\gamma_\mathrm{gs}$, $\gamma_\mathrm{ls}$, and $\gamma_\mathrm{gl}$ between gas and solid, liquid and solid, and gas and liquid, respectively. If $\gamma_\mathrm{gs} < \gamma_\mathrm{ls} + \gamma_\mathrm{gl}$, then $\theta > 0$ (nonzero contact angle), a global equilibrium can be attained, and the liquid thin film is said to partially wet the solid. If on the other hand $\gamma_\mathrm{gs} \ge \gamma_\mathrm{ls}+\gamma_\mathrm{gl}$, then $\theta = 0$ (zero contact angle), a global equilibrium is not attained, and the thin film eventually covers the entire solid (complete-wetting regime).

\medskip

While microscopically Young's law \eqref{young} applies, the apparent macroscopic contact angle is dynamic and in general depends on the flow (for instance through the velocity at the contact line, cf.~\cite{Shikhmurzaev2020} and references therein). The difference is schematically visualized in Figure~\ref{fig:mac_mic}.
\begin{figure}[H]
\begin{subfigure}[t]{.5\textwidth}
  \centering

\tikzset{every picture/.style={line width=0.75pt}}

\begin{tikzpicture}[x=0.75pt,y=0.75pt,yscale=-1,xscale=1]

\draw  [draw opacity=0][fill={rgb, 255:red, 208; green, 205; blue, 205 }  ,fill opacity=1 ] (12.5,138) -- (246.5,138) -- (246.5,170) -- (12.5,170) -- cycle ;
\draw [color={rgb, 255:red, 74; green, 144; blue, 226 }  ,draw opacity=1 ]   (49.5,138) .. controls (132.5,48) and (248.5,66) .. (238.5,63) ;
\draw  [draw opacity=0][dash pattern={on 4.5pt off 4.5pt}] (64.46,116.87) .. controls (68.74,117.82) and (72.63,120.69) .. (74.83,125.02) .. controls (77.22,129.73) and (77.03,135.04) .. (74.78,139.24) -- (61.86,131.6) -- cycle ; \draw  [dash pattern={on 4.5pt off 4.5pt}] (64.46,116.87) .. controls (68.74,117.82) and (72.63,120.69) .. (74.83,125.02) .. controls (77.22,129.73) and (77.03,135.04) .. (74.78,139.24) ;
\draw [color={rgb, 255:red, 254; green, 1; blue, 1 }  ,draw opacity=1,  dashed]   (126.5,32.13) -- (48.5,139) ;
\draw    (11.5,139) -- (246.5,139) ;

\draw (114,148) node [anchor=north west][inner sep=0.75pt]   [align=left] {solid};
\draw (168,98) node [anchor=north west][inner sep=0.75pt]   [align=left] {liquid};
\draw (34,60) node [anchor=north west][inner sep=0.75pt]   [align=left] {gas};
\draw (61,124.6) node [anchor=north west][inner sep=0.75pt]   [align=left] {{\scriptsize K}};

\end{tikzpicture}

  \caption{Schematic of the apparent macroscopic contact angle $K$.}
  \label{Fig:macroscopic}
\end{subfigure}
\begin{subfigure}[t]{.5\textwidth}
  \centering

\tikzset{every picture/.style={line width=0.75pt}}

\begin{tikzpicture}[x=0.75pt,y=0.75pt,yscale=-1,xscale=1]

\draw  [draw opacity=0][fill={rgb, 255:red, 208; green, 205; blue, 205 }  ,fill opacity=1 ] (28.5,123) -- (262.5,123) -- (262.5,155) -- (28.5,155) -- cycle ;
\draw [color={rgb, 255:red, 74; green, 144; blue, 226 }  ,draw opacity=1 ]   (61.5,122.67) .. controls (144.5,116.93) and (135.5,31.27) .. (258.5,20.6) ;
\draw [color={rgb, 255:red, 254; green, 1; blue, 1 }  ,draw opacity=1,  dashed]    (61.5,122.67) .. controls (115.5,125.93) and (180.5,31.27) .. (180.5,31.27) ;
\draw    (28.5,123) -- (263.5,123) ;
\draw  [draw opacity=0][dash pattern={on 4.5pt off 4.5pt}] (89.54,116.52) .. controls (89.48,118.89) and (88.87,121.2) .. (87.78,123.24) -- (74.86,115.6) -- cycle ; \draw  [dash pattern={on 4.5pt off 4.5pt}] (89.54,116.52) .. controls (89.48,118.89) and (88.87,121.2) .. (87.78,123.24) ;
\draw  [draw opacity=0][dash pattern={on 4.5pt off 4.5pt}] (131.33,88.09) .. controls (137.94,88.85) and (144.6,93.64) .. (148.4,101.11) .. controls (152.23,108.67) and (152.11,116.98) .. (148.71,122.78) -- (132.75,109.05) -- cycle ; \draw  [dash pattern={on 4.5pt off 4.5pt}] (131.33,88.09) .. controls (137.94,88.85) and (144.6,93.64) .. (148.4,101.11) .. controls (152.23,108.67) and (152.11,116.98) .. (148.71,122.78) ;

\draw (130,133) node [anchor=north west][inner sep=0.75pt]   [align=left] {solid};
\draw (184,83) node [anchor=north west][inner sep=0.75pt]   [align=left] {liquid};
\draw (50,45) node [anchor=north west][inner sep=0.75pt]   [align=left] {gas};
\draw (134,98.6) node [anchor=north west][inner sep=0.75pt]   [align=left] {{\scriptsize K}};
\draw (95.78,112.24) node [anchor=north west][inner sep=0.75pt]   [align=left] {{\scriptsize k}};

\end{tikzpicture}

  \caption{The previous schematic plot zoomed in near the triple junction. The macroscopic contact angle $K$ and the microscopic contact angle $k$ are shown. The traveling wave is depicted as a dashed line.}
  \label{fig:microscopic}
\end{subfigure}
\caption{ }
\label{fig:mac_mic}
\end{figure}
The main purpose of this note is to investigate the relation between the microscopic and macroscopic contact angle $k$ and $K$, respectively, in the regime of quasi-static motion, where $K$ meets the Cox-Voinov law \cite{Cox1986, Tanner1979, Hocking1992, voinov1977} in an intermediate asymptotic regime which needs to be matched to the bulk solution. This justifies the use of a traveling-wave ansatz, which only captures two asymptotic regimes (Young's and the Cox-Voinov law) and is further explained in \S\ref{subsect_CV}. We expect that this behavior is generic, that is, general solutions exhibit the same behavior in corresponding asymptotic regimes, depending on which addend in the mobility dominates the dynamics. The matched asymptotic expansions of Cox \cite{Cox1986} indicate that the same behavior is to be expected for Stokes flow. Note that significant deviations from the behavior characterized in what follows can be expected if the initial datum dominates the qualitative behavior (see for instance waiting-time phenomena investigated in \cite{ChipotSideris1985,DalPassoGiacomelliGruen2001,DalPassoGiacomelliGruen2003,GiacomelliGruen2006,Fischer2014,FischerMatthes2021} in the complete-wetting regime and references therein), or if the film thickness decreases below the slip length $\lambda$, so that the term $h^n$ in \eqref{TFE_PDE} is dominating (see for instance self-similar asymptotics investigated in \cite{CarrilloToscani2002,BernoffWitelski2002,CarlenUlusoy2007,CarlenUlusoy2014,Gnann2015,Seis2018} in the complete-wetting regime  and references therein). Additionally note that for very thin films (at the order of only a few fluid molecules thickness), thermal fluctuations modelled by an additional stochastic forcing play a role (see \cite{DavidovitchMoroStone2005,GruenMeckeRauscher2006}, where the corresponding stochastic thin-film equation was proposed first). Rigorous analytic results on the latter model can be found in \cite{FischerGruen2018,GessGnann2020,DareiotisGessGnannGruen2021,GruenKlein2021,MetzgerGruen2021,Sauerbrey2021}.

\section{Setting and main result}\label{chapter3}
In this section, the ordinary boundary-value problem describing the traveling wave is formulated and suitably transformed. Afterwards our main theorem is stated. Note that the transformations presented in the sequel are similar to those used in \cite[\S1]{GGO}, where complete-wetting boundary conditions have been treated.

\subsection{The traveling-wave problem of the thin-film equation}
Using the travling-wave ansatz $h = H$, where $H$ only depends on $x = y + Vt$ and $V$ is the constant and finite velocity of the film, and assuming that $\rbar{Y}{t = 0} = 0$ by translation invariance, the above problem \eqref{TFE} can be rewritten in terms of the third-order ordinary differential equation (ODE)
\begin{subequations}\label{EQ_3.2}
\begin{align}
(H^2 +\lambda^{3-n}H^{n-1})\tfrac{\d^3 H}{\d x^3} &= -V && \text{in } (0,\infty) \label{TFE_unscaled}
\end{align}
with boundary conditions
\begin{align}
    H &= 0 && \text{at } x=0, \label{BC_TFE_unscaled_1} \\
    \tfrac{\d H}{\d x} &= k && \text{at } x=0, \label{BC_TFE_unscaled_2} \\
    (H^2 +\lambda^{3-n}H^{n-1})\tfrac{\d^3 H}{\d x^3} &= -V && \text{at } x = 0. \label{BC_TFE_unscaled_3}
\end{align}
\end{subequations}
Indeed, the boundary conditions \eqref{BC_TFE_unscaled_1}, \eqref{BC_TFE_unscaled_2}, and \eqref{BC_TFE_unscaled_3} follow trivially from the boundary conditions \eqref{BC:1}, \eqref{BC:2}, and \eqref{BC:3}, respectively. Furthermore, the partial differential equation (PDE) \eqref{TFE_PDE} turns into the ODE
\begin{align*}
V \tfrac{\d H}{\d x} + \tfrac{\d}{\d x}\left((H^3+\lambda^{3-n}H^n)\tfrac{\d^3H}{\d x^3}\right) &= 0 && \text{in } (0,\infty).
\end{align*}
Integrating in $x$ leads to
\begin{align*}
V H + (H^3+\lambda^{3-n}H^n)\tfrac{\d^3H}{\d x^3} &= c && \text{in } (0,\infty),
\end{align*}
where $c$ is a constant. The boundary conditions \eqref{BC_TFE_unscaled_1} and \eqref{BC_TFE_unscaled_3} entail $c = 0$, so that \eqref{TFE_unscaled} is obtained by dividing through $H$.

\medskip

Under the additional assumption of vanishing curvature in the bulk, that is,
\begin{subequations}[resume]
\begin{align}
\tfrac{\d^2H}{\d x^2} &\to 0 &&\text{as } x\to\infty.\label{BC_TFE_unscaled_4}
\end{align}
\end{subequations}
Chiricotto and Giacomelli have found in \cite{CG} that the boundary-value problem \eqref{TFE_unscaled} for $n = 2$ has a unique classical solution $H = H_\mathrm{CG}$ which is three times continuously differentiable in $x > 0$ with $H_\mathrm{CG}$ and $\frac{\d H_\mathrm{CG}}{\d x}$ continuous in $x \ge 0$. Their reasoning also applies to $n \in (0,3)$, which is why we can assume from hereon that a unique $H = H_\mathrm{CG}$ solving \eqref{EQ_3.2} for $n \in (0,3)$ exists. For the reader's convenience, we give a streamlined version of the existence and uniqueness proof in \cite{CG} in a different set of variables in Theorem~\ref{th:ex_un} in Appendix~\ref{app:ex_un}.

\medskip

Note that by applying the scalings
\begin{equation}\label{rescaling_x,H}
H \mapsto \lambda H, \qquad x \mapsto \left(3V\right)^{-\frac{1}{3}} \lambda x, \qquad \text{and} \qquad k \mapsto \left(3V\right)^{\frac 1 3} k,
\end{equation}
we may without loss of generality assume $\lambda = 1$ and $V=\frac{1}{3}$, so that equations~\eqref{EQ_3.2} turn into finding $H$ such that
\begin{subequations}\label{scaled_partial_wetting}
\begin{align}
    \left(H^2 +H^{n-1}\right) \tfrac{\d^3 H}{\d x^3} &=-\tfrac{1}{3} && \text{for $x>0$}, \label{diff_scaled}\\
    H &= 0 && \text{at $x=0$}, \label{bc_partial_wetting_1}\\
    \tfrac{\d H}{\d x} &= k && \text{at $x=0$}, \label{scaled_angle} \\
    \tfrac{\d^2 H}{\d x^2} &\to 0 && \text{as $x \to \infty$},\label{bc_partial_wetting_2}
\end{align}
\end{subequations}
which is uniquely solved by $H = H_\mathrm{CG}$.

\subsection{The Cox-Voinov law}\label{subsect_CV}
Recall that we have chosen $n \in (0,3)$, so that as $x \to \infty$ the term $H^{2}$ dominates $H^{n-1}$ in equation~\eqref{diff_scaled}. This is why the expected behavior of the differential equation \eqref{diff_scaled} is determined by
\begin{align}\label{diff_to_infty}
    H^2 \tfrac{\d^3 H}{\d x^3} &= - \tfrac 1 3 && \mbox{as $x \to \infty$.}
\end{align}
Then, it can be easily recognized that \eqref{diff_to_infty} is approximately solved by the asymptotic
\begin{align}\label{tanner_scaled}
    H &= x(\ln x)^{\frac{1}{3}}(1+\landau(1)) && \text{as $x\to\infty$}.
\end{align}
In fact, an implicit solution of \eqref{diff_to_infty} in terms of Airy functions was found by Duffy and Wilson in \cite{DuffyWilson1997}, from which the asymptotic \eqref{tanner_scaled} can be derived. Formally differentiating \eqref{tanner_scaled} with respect to $x$, raising it to the power of three, and reverting the normalization of the speed $V$ gives
\begin{align}\label{Tanner}
    \left(\tfrac{\d H}{\d x}\right)^3 &= 3V \ln x (1+\landau(1)) && \text{as $x\to\infty$}.
\end{align}
Again, we note that equation~\eqref{Tanner} can be made rigorous using \cite{DuffyWilson1997}. Because the lubrication approximation assumes small slopes, $\frac{\d H}{\d x}$ is in this approximation, as $x \to \infty$, equal to the macroscopic contact angle. Hence, this asymptotic implies that the cube of the macroscopic contact angle is, up to a logarithmic correction, proportional to the speed of the free boundary. This will be referred to as \emph{the Cox-Voinov law} \cite{Cox1986, voinov1977}, in what follows, though the relation between microscopic and macroscopic contact angle has been analyzed also by Tanner \cite{Tanner1979} and Hocking \cite{Hocking1983}. Corresponding rigorous results regarding intermediate-in-time asymptotics, known as Tanner's law \cite{Tanner1979}, can be found in \cite{GiacomelliOtto2002,DelgadinoMellet2021}.

\medskip

Note that the subsequent results are limited since we are considering a droplet that infinitely extends to $x \to \infty$. In realistic situations, the apparent/macroscopic contact angle can be measured at an inflection point close to the contact line (point of maximum slope, see \cite{Tanner1979}). Thus, the Cox-Voinov law is only an intermediate asymptotic and needs to be matched to a bulk solution (see \cite{EggersStone2004} for matched-asymptotics arguments). Carrying this out rigorously is rather delicate and exceeds the presentation of this note.

\medskip

For the subsequent results, it is important to note that the solution to \eqref{diff_to_infty} is invariant under translation in $x$, that is, replacement of $x \mapsto x + c$ for any $c \in \R$, and the scaling transformation $(x,H) \mapsto (B x, B H)$ for any $B > 0$, which leads to a two-parameter family of solutions meeting the asymptotic \eqref{tanner_scaled}. The translation invariance will be removed by a suitable coordinate transformation in the following section. The remaining parameter $B$ will be used in order to rigorously match the asymptotic \eqref{Tanner} to the microscopic Young angle $k$ of the unique classical solution to \eqref{scaled_partial_wetting}. The precise mathematical result is given in Theorem~\ref{main_thm} in \S\ref{sec:main} below.
 
\subsection{Coordinate transformation}
Obviously, equation~\eqref{diff_scaled} is translation-invariant in $x$.
For the classical solution $H = H_\mathrm{CG}$ of problem~\eqref{scaled_partial_wetting}, we also have the following properties:
\begin{enumerate}
    \item\label{it:hg0} It holds $H_\mathrm{CG}>0$ for all $x>0$. This is true because $H_\mathrm{CG} > 0$ for $0 < x \ll 1$ due to \eqref{scaled_angle} and $k > 0$. On the other hand, continuity of $H_\mathrm{CG}$ and $\frac{\d^3 H_\mathrm{CG}}{\d x^3}$, and \eqref{diff_scaled} prevent $H_\mathrm{CG}$ from becoming zero, which yields $H_\mathrm{CG} > 0$ for all $x > 0$.
    \item\label{it:d3hl0} We have $\frac{\d^3 H_\mathrm{CG}}{\d x^3} < 0$ for all $x>0$ by \eqref{diff_scaled} and \eqref{it:hg0}.
    \item\label{it:d2hl0} We get $\frac{\d^2 H_\mathrm{CG}}{\d x^2} >0 $ for all $x>0$ by \eqref{bc_partial_wetting_2} and \eqref{it:d3hl0}.
    \item We have $\frac{\d H_\mathrm{CG}}{\d x} >0$ for all $x>0$ by  \eqref{bc_partial_wetting_1}, $k > 0$, and \eqref{it:d2hl0}.
\end{enumerate}
The above shows that $H_\mathrm{CG}$ is a strictly increasing function, so that \eqref{diff_scaled} can be rewritten in terms of $x = x_\mathrm{CG}$ as a function of $H$, thus removing the translation invariance in $x$ and leading to a second-order ODE instead of the third-order ODE \eqref{diff_scaled}. This equation, however, includes $x_\mathrm{CG}$, $\frac{\d x_\mathrm{CG}}{\d H}$, and $\frac{\d^2 x_\mathrm{CG}}{\d H^2}$, which makes it inconvenient for a monotonicity argument. Instead, we opt for the choice
\begin{equation}\label{coordinate_transform}
    \psi := \left(\tfrac{\d H}{\d x}\right)^2 = \left(\tfrac{\d x}{\d H}\right)^{-2} > 0 \text{ as a function of } H
\end{equation}
in what follows. Then, problem~\eqref{scaled_partial_wetting} turns into finding $\psi$ such that
\begin{subequations}\label{main_sys}
\begin{align}\label{main_ode}
    \tfrac{\d^2\psi}{\d H^2} + \tfrac{2}{3}(H^2+H^{n-1})^{-1} \psi^{-\frac{1}{2}} &= 0 && \text{for $H>0$},
\end{align}
\end{subequations}
where the boundary conditions are given by
\begin{subequations}[resume]
\begin{align}
    \psi &= k^2 &&\text{at $H=0$}, \label{BC_main_sys_1}\\
    \tfrac{\d\psi}{\d H} &\to 0 &&\text{as $H \to \infty$.} \label{BC_main_sys_2}
\end{align}
\end{subequations}
Indeed, we have
\begin{align*}
\tfrac{\d\psi}{\d H} \stackrel{\eqref{coordinate_transform}}{=} 2 \tfrac{\d H}{\d x} \tfrac{\d^2H}{\d x^2} \tfrac{\d x}{\d H} = 2 \tfrac{\d^2 H}{\d x^2}, \qquad
\tfrac{\d^2\psi}{\d H^2} = 2 \tfrac{\d^3H}{\d x^3} \tfrac{\d x}{\d H} \stackrel{\eqref{coordinate_transform}}{=} 2 \tfrac{\d^3H}{\d x^3} \psi^{-\frac 1 2},
\end{align*}
and thus
\[
\tfrac{\d^2\psi}{\d H^2} \stackrel{\eqref{diff_scaled}}{=} - \tfrac 2 3 \left(H^2+H^{n-1}\right)^{-1} \psi^{-\frac 1 2},
\]
which yields \eqref{main_ode}. On the other hand, the boundary conditions \eqref{BC_main_sys_1} and \eqref{BC_main_sys_2} follow directly from the definition of $\psi$ in \eqref{coordinate_transform} and the boundary conditions of \eqref{bc_partial_wetting_1}--\eqref{bc_partial_wetting_2}. The main result of Chiricotto and Giacomelli in \cite{CG} implies that \eqref{main_sys} has a unique classical solution $\psi = \psi_\mathrm{CG}$ being twice continuously differentiable in $H > 0$ and right-continuous at $H = 0$. The result and proof generalized to $n \in (0,3)$ and adapted to the system \eqref{main_sys} can be found in Appendix~\ref{app:ex_un}, Theorem~\ref{th:ex_un}.

\subsection{The Cox-Voinov law in new coordinates\label{sec:tanner_psi}}
With help of \eqref{coordinate_transform}, the leading-order equation \eqref{diff_to_infty} can now be rephrased as
\begin{subequations}\label{problem_tanner}
\begin{align}\label{diff_to_infty_psi}
    \tfrac{\d^2\psi}{\d H^2} + \tfrac{2}{3} H^{-2} \psi^{-\frac 1 2} &= 0 && \text{for large $H>0$}
\end{align}
with Cox-Voinov asymptotic
\begin{align}\label{tanner_scaled_psi}
    \psi &= (\ln H)^{\frac 2 3}(1+\landau(1)) && \text{as $H\to\infty$}.
\end{align}
The family of solutions to \eqref{diff_to_infty_psi} meeting \eqref{tanner_scaled_psi} is now one-parametric because of the scaling invariance $H \mapsto B H$ for any $B > 0$. It is proved in \cite[Proposition~3.1]{GGO} that problem~\eqref{diff_to_infty_psi} has a unique solution $\psi = \psi_\mathrm{CV}$ being twice continuously differentiable for $H > 0$ large if we additionally demand the refined asymptotic
\begin{align}\label{tanner_scaled_psi_t}
    \psi^{\frac 3 2} &= \ln H - \tfrac 1 3 \ln\left(\ln H\right) + \landau(1) && \text{as $H\to\infty$}.
\end{align}
\end{subequations}
We select this solution $\psi_\mathrm{CV}$ from now on.

\subsection{The main result\label{sec:main}}
The rest of this paper is devoted to proving the following result, giving a precise characterization of the asymptotic regimes as $H \to \infty$ and $H \downarrow 0$ and their dependence on the parameters $n$ (mobility exponent) and $k$ (microscopic contact angle).
\begin{theorem}\label{main_thm}
Suppose $n\in(0,3)$ and $k > 0$. The unique solution $\psi = \psi_\mathrm{CG}$ to \eqref{main_sys} being twice continuously differentiable in $H > 0$ and right-continuous at $H = 0$, has the following asymptotic regimes:
\begin{enumerate}
\item There exists a real parameter $B > 0$ and a function $R_\infty$ of $H$ such that
\begin{align}\label{main_thm_eq1}
    \psi_\mathrm{CG} &= \rbar{\psi_\mathrm{CV}}{H \mapsto BH} (1+R_\infty) && \mbox{for } H > 0 \mbox{ sufficiently large},
\end{align}
where $C > 0$ is a constant, $\psi_\mathrm{CV}$ is chosen as in \S\ref{sec:tanner_psi}, and
\begin{align*}
    R_\infty &= \order\left((\ln(H))^{-1}H^{-(3-n)}\right) && \text{as $H\to\infty$}.
\end{align*}
The parameter $B$ and the correction $R_\infty$ are continuously differentiable functions of $k > 0$.
\item It holds
\begin{align}\label{main_thm_eq2}
    \psi_\mathrm{CG} &= k^2(1+\mu) && \text{as $H\downarrow0$},
\end{align}
where $\mu$ has the following properties:
\begin{itemize}
\item[i)] For $n \in (0,3) \setminus \left\{3-\frac 1 m \colon m \in \N\right\}$ (non-resonant case) it holds $\mu = \rbar{v}{(\zeta,\varrho) = \left(H,H^{3-n}\right)}$ as $H \downarrow 0$, where $v$ is analytic in $(\zeta,\varrho)$ around $(\zeta,\varrho) = (0,0)$ and smooth in $k > 0$ with $\rbar{v}{(\zeta,\varrho) = (0,0)} = 0$.
\item[ii)] For $n = 3-\frac 1 m$ with $m \in \N$ (resonant case) it holds $\mu = \rbar{v}{(\zeta,\varrho,\sigma) = \left(H,H^{3-n},H \ln H\right)}$ as $H \downarrow 0$, where $v$ is analytic in $(\zeta,\varrho,\sigma)$ around $(\zeta,\varrho,\sigma) = (0,0,0)$ and smooth in $k > 0$ with $\rbar{v}{(\zeta,\varrho,\sigma) = (0,0,0)} = 0$.
\end{itemize}
\end{enumerate}
\end{theorem}
We emphasize that Theorem~\ref{main_thm} is the analogue of \cite[Theorem~2.1]{GGO} in which complete-wetting boundary conditions ($k = 0$) are studied. The asymptotic \eqref{main_thm_eq1} of Theorem~\ref{main_thm} contains information on the apparent (macroscopic) contact angle. Indeed, because the parameter $B$ and the remainder $R_\infty$ depend continuously differentiably on the microscopic contact angle $k > 0$, we obtain from \eqref{coordinate_transform}, \eqref{tanner_scaled_psi_t}, and \eqref{main_thm_eq1} that
\begin{align*}
\left(\tfrac{\d H}{\d x}\right)^3 &= \ln(B H) - \tfrac 1 3 \ln\left(\ln H\right) + \landau(1) && \mbox{as } H \to \infty,
\end{align*}
where $B > 0$ and $\landau(1)$ depend continuously differentiably on $k > 0$. This separable ODE yields
\begin{align*}
H &= x \left(\ln(Bx)\right)^{\frac 1 3} \left(1+\landau(1)\right) && \mbox{as } x \to \infty,
\end{align*}
so that we obtain
\begin{align*}
\left(\tfrac{\d H}{\d x}\right)^3 &= \ln(B x) + \landau(1) && \mbox{as } x \to \infty,
\end{align*}
which after undoing the scalings \eqref{rescaling_x,H} yields
\begin{align*}
\left(\tfrac{\d H}{\d x}\right)^3 &= 3 V \ln\left(B (3V)^{\frac 1 3} \lambda^{-1} x\right) + \landau(1) && \mbox{as } x \to \infty,
\end{align*}
where $B > 0$ and $\landau(1)$ depend in all instances continuously differentiably on $k > 0$. In conclusion, we have shown that the macroscopic contact angle depends continuously differentiably on the microscopic contact angle and thus by Young's law \eqref{young} on the physically adjustable surface tensions acting at the interfaces. This is the novelty compared to \cite{GGO}, where $k = 0$ was considered and the dependence of the asymptotic as $H \to \infty$ on the parameter $n \in \left(\frac 3 2, \frac 7 3\right)$ (mobility exponent) was studied. Further note that Eggers in \cite{Eggers2004} has studied the same problem and by matched asymptotics has determined an expansion of $B$ in terms of the inverse of a rescaled capillary number (proportional to the velocity $V$ of the contact line divided by the cube $k^3$ of the microscopic contact angle). Our result provides a rigorous justification of an existence of such an expansion to leading order. Further note that we strongly believe that the arguments provided in the present note can be lifted to prove smoothness of $B$ and $R_\infty$ in Theorem~\ref{main_thm} in $k > 0$. However, this would require to revisit many of the technical steps carried out in \cite[\S5]{GGO} in order to prove smoothness in $B > 0$ of the solution manifold meeting the Cox-Voinov law, characterized in \cite[Proposition~3.1]{GGO} (Proposition~\ref{prop:manifold_bulk} in this note), while not providing any significantly new mathematical insights.

\medskip

The asymptotics \eqref{main_thm_eq2}, on the other hand, give us information about the behavior of the solution close to the contact line (microscopic regime). We recognize that the value of $\psi_\mathrm{CG}$ as $H\downarrow 0$ is equal to $k^2$ with a precisely characterized correction continuously differentiably depending on $k > 0$. In particular, on noting that $\frac{\d^2 H_\mathrm{CG}}{\d x^2}$ gives up to a constant the pressure at the interface (it is proportional to the curvature which in lubrication approximation is merely the second derivative of the profile in the spatial variable), the derivative $\frac{\d \psi_\mathrm{CG}}{\d H}$ gives up to a constant the pressure, that is, we obtain the singularity
\begin{align*}
\tfrac{\d\psi_\mathrm{CG}}{\d H} &= k^2 \rbar{\partial_\zeta v}{(\zeta,\varrho) = \left(H,H^{3-n}\right)} + (3-n) k^2 \rbar{\partial_\varrho v}{(\zeta,\varrho) = \left(H,H^{3-n}\right)} H^{2-n} && \mbox{as } H \downarrow 0
\end{align*}
for $n \in (0,3) \setminus \left\{3-\frac 1 m \colon m \in \N\right\}$ and
\begin{align*}
\tfrac{\d\psi_\mathrm{CG}}{\d H} &= k^2 \rbar{\partial_\zeta v}{(\zeta,\varrho,\sigma) = \left(H,H^{3-n},H \ln H\right)} + (3-n) k^2 \rbar{\partial_\varrho v}{(\zeta,\varrho,\sigma) = \left(H,H^{3-n},H \ln H\right)} H^{2-n} \\
&\phantom{=} + k^2 \rbar{\partial_\sigma v}{(\zeta,\varrho,\sigma) = \left(H,H^{3-n},H \ln H\right)} \left(1+\ln H\right) && \mbox{as } H \downarrow 0
\end{align*}
for $n = 3 - \frac 1 m$ with $m \in \N$. Here, we have $v := b \zeta + \rbar{w}{\xi = b \zeta}$, where $b = b_\mathrm{CG} \in \R$ is a uniquely determined parameter matching the solution to the Cox-Voinov manifold characterized by the asymptotics \eqref{main_thm_eq1} and $w$ is uniquely determined in Propositions~\ref{prop:nonresonant} and \ref{prop:resonant} in \S\ref{sec:fixed} below. Similar singular expansions have been found in \cite[Theorems~3.2 and 3.3]{BGK} in case of source-type self-similar solutions with dynamic contact angle condition and in \cite{Knuepfer2011,Knuepfer2015,Knuepfer_erratum} in case of the thin-film equation with homogeneous mobility and partial-wetting boundary conditions. In case of partial wetting, we also refer to \cite{Degtyarev2017} for existence, uniqueness, and regularity in higher dimensions, to \cite{Esselborn2016,MajdoubMasmoudiTayachi2020} for existence, uniqueness, and stability, and to \cite{Otto1998,BertschGiacomelliKarali2005,Mellet2015} for existence results on weak solutions.

\subsection{Outline}
The rest of the paper is devoted to the proof of Theorem~\ref{main_thm}. This relies on one hand on a precise characterization of the solution manifold near the contact line (cf.~\S\ref{sec:contact}) using dynamical-systems techniques and the matching of this solution manifold with the solution manifold as $H \to \infty$ as characterized in \cite[Proposition~3.1]{GGO} (cf.~Proposition~\ref{prop:manifold_bulk}). This matching argument is carried out in \S\ref{sec:proof_main}. In Appendix~\ref{app:ex_un} we give a streamlined version of the existence and uniqueness proof of \cite{CG} for the system \eqref{main_sys} instead of \eqref{scaled_partial_wetting}.

\section{The solution manifold near the contact line\label{sec:contact}}
Note that the construction of a solution manifold at the contact line is in part based on the analysis in \cite[\S4.2--4.4]{BGK} in which partial-wetting boundary conditions for the source-type self-similar solution with homogeneous mobility are treated. Our reasoning is different in that we choose to study a dynamical system that is changed compared to \cite[\S4.2--4.4]{BGK} with the advantage that the contact line corresponds to a hyperbolic fixed point. Furthermore, we additionally discuss the smooth dependence on the parameter $k > 0$.

\subsection{Reformulation as a dynamical system}\label{chapter4}
In this section, a dynamical system will be formulated to characterize the error between $\psi$ solving \eqref{main_ode} and \eqref{BC_main_sys_1} and the squared microscopic contact angle $k^2$ as $H\downarrow0$.
\subsubsection{Coordinate transformations}
We first apply the coordinate transformation
\begin{subequations}\label{s_mu}
\begin{equation}\label{s_h}
s:=\ln H,
\end{equation}
which shifts the contact line $H=0$ to $s=-\infty$. Secondly, we introduce the new dependent variable $\mu$ with
\begin{equation}\label{mu}
\mu := \tfrac{\psi}{k^2}-1,
\end{equation}
\end{subequations}
determining the error between $\psi$ and $k^2$. On noting that $\frac{\d}{\d H} \stackrel{\eqref{s_h}}{=} e^{-s} \frac{\d}{\d s}$, the transformations \eqref{s_mu} turn problem~\eqref{main_sys} into
\begin{subequations}\label{mu_sys}
\begin{align}
    \tfrac{\d^2\mu}{\d s^2} - \tfrac{\d\mu}{\d s} + \tfrac{2}{3 k^3(1+e^{-(3-n)s})} \left(1+\mu\right)^{- \frac 1 2} &=0 && \text{for $s\in \R$,} \label{ODE_s} \\
    \mu &\to 0 && \mbox{as $s \to - \infty$,} \label{BC_ode_s} \\
    e^{-s} \tfrac{\d\mu}{\d s} &\to 0 && \mbox{as $s \to \infty$,} \label{BC_ode_s_inf}
\end{align}
\end{subequations}
which is uniquely solved by $\mu = \mu_\mathrm{CG}$ given by \eqref{s_mu} with $\psi = \psi_\mathrm{CG}$.

\subsubsection{The dynamical system}
Equation \eqref{ODE_s} will now be reformulated as an autonomous three-dimensional continuous dynamical system using the functions
\begin{equation}\label{def_r_q_p}
r := e^{\frac{3-n}{3} s}, \qquad q := e^{- \frac{3-n}{3} s} \mu \qquad \text{and} \qquad p := e^{- \frac{3-n}{3} s} \tfrac{\d\mu}{\d s}.
\end{equation}
If $\mu = \mu_\mathrm{CG}$ we write $(r,q,p) = \left(r_\mathrm{CG},q_\mathrm{CG},p_\mathrm{CG}\right)$. The dynamical system becomes
\begin{subequations}\label{ds_f}
\begin{equation}\label{ds}
\tfrac{\d}{\d s} \left(r, q, p\right)
= F,
\end{equation}
where
\begin{equation}\label{def_f}
F := \left(\tfrac{3-n}{3} r, - \tfrac{3-n}{3} q + p, \tfrac n 3 p - \tfrac{2}{3 k^3} \tfrac{r^2}{1+r^3} (1+r q)^{-\frac{1}{2}}\right).
\end{equation}
\end{subequations}
It can be easily verified that for our choice $n \in (0,3)$ the point $(0,0,0)$ is the unique fixed point of the system \eqref{ds}. In the next lemma we will see that any solution $(r,q,p)$, which under the transformations \eqref{def_r_q_p} meets \eqref{ODE_s} and \eqref{BC_ode_s}, converges to this fixed point as $s\to -\infty$ and we additionally characterize the asymptotic behavior.
\begin{lemma}\label{lem:conv_cg}
Suppose $k > 0$, $n \in (0,3)$, that $\mu$ is an in $s \in \R$ twice continuously differentiable solution to \eqref{ODE_s} and \eqref{BC_ode_s}, and let $(r,q,p)$ be defined by \eqref{def_r_q_p}. Then it holds
\begin{subequations}\label{conv_cg}
\begin{align}
r &= e^{\frac{3-n}{3} s} && \mbox{for all $s \in \R$,} \label{conv_rcg} \\
q &= \begin{cases} \order\left(e^{\frac n 3 s}\right) &\text{ for $0 < n < 2$}, \\
- \tfrac{2}{3 k^3} s e^{\frac n 3 s} (1+\landau(1)) &\text{ for $n=2$}, \\
\tfrac{2}{3 (3-n) (n-2) k^3} e^{\frac 2 3 (3-n) s}(1+\landau(1)) &\text{ for $2<n<3,$}
\end{cases} && \mbox{as $s \to - \infty$,} \label{conv_qcg} \\
p &= \begin{cases} \order\left(e^{\frac n 3 s}\right) &\text{ for $0 < n < 2$}, \\
- \tfrac{2}{3 k^3} s e^{\frac n 3 s} (1+\landau(1)) &\text{ for $n=2$}, \\
\tfrac{2}{3 (n-2) k^3} e^{\frac 2 3 (3-n) s}(1+\landau(1)) &\text{ for $2<n<3,$}
\end{cases} && \mbox{as $s \to - \infty$,} \label{conv_pcg}
\end{align}
\end{subequations}
so that in particular $\left(r,q,p\right) \to (0,0,0)$ as $s\to -\infty$.
\end{lemma}
\begin{proof}
We have $r \stackrel{\eqref{def_r_q_p}}{=} e^{\frac{3-n}{3} s}$ so that \eqref{conv_rcg} immediately follows.

\medskip

In order to determine the asymptotic behavior of $p$, we compute
\begin{equation}\label{rel_p_psi}
p \stackrel{\eqref{def_r_q_p}}{=} e^{- \frac{3-n}{3} s} \tfrac{\d\mu}{\d s} \stackrel{\eqref{mu}}{=} \tfrac{e^{- \frac{3-n}{3} s}}{k^2} \tfrac{\d\psi}{\d s} \stackrel{\eqref{s_h}}{=} \tfrac{e^{\frac n 3 s}}{k^2} \tfrac{\d\psi}{\d H}.
\end{equation}
Hence, the asymptotic of $p$ is determined by the asymptotic of $\frac{\d\psi}{\d H}$. Therefore, note that from \eqref{mu} and \eqref{BC_ode_s} it follows that $\psi = k^2 (1+\landau(1))$ as $H\downarrow0$ and equation~\eqref{main_ode} (which by virtue of \eqref{s_mu} is equivalent to \eqref{ODE_s}) gives
\begin{align*}
\tfrac{\d^2\psi}{\d H^2} &= - \tfrac{2}{3}(H^2+H^{n-1})^{-1} \psi^{-\frac 1 2} = - \tfrac{2}{3 k} H^{1-n} (1+\landau(1)) && \mbox{as $H\downarrow0$}.
\end{align*}
In order to obtain an expression for $\frac{\d\psi}{\d H}$, take $\eps > 0$ and write
\begin{align*}
\tfrac{\d\psi}{\d H} &= \rbar{\tfrac{\d\psi}{\d H}}{H=\eps} - \int_H^\eps \rbar{\tfrac{\d^2\psi}{\d H^2}}{H=\tilde H} \d\tilde H = \rbar{\tfrac{\d\psi}{\d H}}{H=\eps} + \tfrac{2}{3k} (1+\landau(1))\int_H^\varepsilon\tilde H^{1-n}d\tilde H \\
&=\begin{cases} C(\varepsilon)-\frac{2}{3 (2-n) k} H^{2-n} (1+\landau(1)) &\text{ as $H\downarrow 0$ for $n\ne2$,} \\
C(\varepsilon) - \tfrac{2}{3 k} (\ln H) (1+\landau(1)) &\text{ as $H\downarrow 0$ for $n=2$,}
\end{cases}
\end{align*}
where $C(\varepsilon)$ is a constant only depending on $\eps$. This implies
\begin{align*}
\tfrac{\d\psi}{\d H} = \begin{cases} C(\varepsilon)(1+\landau(1)) &\text{ as $H\downarrow 0$ for $0<n<2$}, \\
- \tfrac{2}{3 k} \ln H (1+\landau(1)) &\text{ as $H\downarrow 0$ for $n=2$}, \\
\tfrac{2}{3 (n-2) k} H^{2-n}(1+\landau(1)) &\text{ as $H\downarrow 0$ for $2<n<3,$}
\end{cases}
\end{align*}
so that because of \eqref{s_h} and \eqref{rel_p_psi} we obtain \eqref{conv_pcg}.

\medskip

Finally, since
\[
q \stackrel{\eqref{BC_ode_s}, \eqref{def_r_q_p}}{=} e^{- \frac{3-n}{3} s} \int_{-\infty}^s e^{\frac{3-n}{3} \tilde s} \rbar{p}{s=\tilde s} \d\tilde s,
\]
we obtain \eqref{conv_qcg} from \eqref{conv_pcg}.
\end{proof}
%

\subsection{Characterization of the unstable manifold}
\subsubsection{Hyperbolicity and linearization}
Equation \eqref{ds} can be linearized around the fixed point $(r,q,p) = (0, 0, 0)$, resulting in
\[
    DF \stackrel{\eqref{def_f}}{=}
    \begin{pmatrix} \frac{3-n}{3} & 0 & 0 \\
    0 & -\frac{3-n}{3} & 1 \\
    -\frac{2}{3 k^3} \frac{2 r - r^4}{(1+r^3)^2} (1+rq)^{-\frac 1 2} + \frac{1}{3 k^3} \frac{r}{1+r^3} \frac{rq}{(1+rq)^{\frac 32}} & \frac{1}{3k^3} \frac{r^3}{1+r^3} (1+rq)^{-\frac 3 2} & \frac{n}{3}
    \end{pmatrix},
\]
so that
\[
     \rbar{DF}{(r,q,p) = (0,0,0)} =
     \begin{pmatrix} \frac{3-n}{3} & 0 & 0\\
    0&-\frac{3-n}{3}&1\\
    0&0&\frac{n}{3}
    \end{pmatrix},
\]
where $DF$ denotes the Jacobian matrix of $F$ evaluated in $(0, 0, 0)$. The eigenvalues are distinct and equal to $\frac{3-n}{3}$, $-\frac{3-n}{3}$, and $\frac n 3$, so that because of $n \in (0,3)$ the fixed point $(r,q,p) = (0,0,0)$ is hyperbolic with two-dimensional unstable manifold $M^-$ and one-dimensional stable manifold $M^+$. Note that hyperbolicity is ensured by including the factors $e^{- \frac{3-n}{3} s}$ in the definitions of $r$, $q$, and $p$, as otherwise the system would have infinitely many non-hyperbolic fixed points.

\medskip

The linearized system can be diagonalized, that is,
\begin{equation}\label{diagonal}
\rbar{DF}{(r,q,p) = (0,0,0)} = \begin{pmatrix} 1 & 0 & 0 \\
0 & 1 & 1 \\
0 & 0 & 1
\end{pmatrix} \begin{pmatrix} \frac{3-n}{3} & 0 & 0 \\ 0 & - \frac{3-n}{3} & 0 \\ 0 & 0 & \frac n 3
\end{pmatrix} \begin{pmatrix} 1 & 0 & 0 \\
0 & 1 & -1 \\
0 & 0 & 1
\end{pmatrix}.
\end{equation}
The representation \eqref{diagonal} is convenient in order to characterize the unstable manifold.

\subsubsection{The unstable manifold}
%
\begin{lemma}\label{lem:unstable}
For $n\in(0,3)$ and $k > 0$, let $\mu$ be an in $s \in \R$ twice continuously differentiable solution to \eqref{ODE_s} and \eqref{BC_ode_s} and let $(r,q,p)$ be defined by \eqref{def_r_q_p}. Then $(r,q,p)$ lies on the unstable manifold $M^{-}$ of the fixed point $(0,0,0)$ of the dynamical system \eqref{ds_f}. The unstable manifold $M^-$ can be parameterized by $p = p^-$, where $p^-$ as a function of $(r,q,k)$ is analytic in $(r,q)$ in a neighborhood of $(r,q) = 0$ meeting the partial differential equation
\begin{equation}\label{unstable_eq}
\left(r \partial_r - q \partial_q - \tfrac{n}{3-n}\right) p^- + \tfrac{3}{3-n} p^- \partial_q p^- = - \tfrac{2}{(3-n) k^3} \tfrac{r^2}{1+r^3} \left(1+rq\right)^{-\frac 1 2}
\end{equation}
and smooth in $k > 0$ with
\begin{subequations}\label{partial_p-}
\begin{align}
p^- &= 0 && \mbox{at $(r,q) = (0,0)$}, \label{partial_p-_0} \\
\partial_r p^{-} &= 0 && \mbox{at $(r,q) = (0,0)$}, \label{partial_p-_1r} \\
\partial_q p^- &= 1 && \mbox{at $(r,q) = (0,0)$}, \label{partial_p-_1q} \\
\partial_r^2 p^- &= - \tfrac{4}{3 k^3 (3-n)} && \mbox{at $(r,q) = (0,0)$}, \label{partial_p-_2r} \\
\partial_r^j \partial_q^\ell p^- &= 0 && \mbox{at $(r,q) = (0,0)$ for $(j,\ell) \in \N_0^2$ with $j \le \ell-2$.} \label{partial_p-_jrlq} 
\end{align}
\end{subequations}
\end{lemma}
\begin{proof}
The tangent space to the unstable manifold $M^-$ at $(0,0,0)$ is spanned by the vectors (cf.~\eqref{diagonal})
\[
v_1 := (1, 0, 0) \qquad \mbox{and} \qquad v_2 := (0, 1, 1).
\]
A vector perpendicular to $v_1$ and $v_2$ is given by
\[
v_1 \times v_2 = (0 , -1, 1),
\]
so that the tangent space to $M^{-}$ at $(0,0,0)$ is given by
\begin{equation}\label{tangent_space}
p = q.
\end{equation}
Hence, $M^-$ can be parameterized by $p = p^-$, where $p^{-}$ is a function of $(r,q,k)$. The analyticity of $F$ in $(r,q,p) = (0,0,0)$ (cf.~\eqref{def_f}) implies that $M^{-}$ is analytic in a neighboorhood of $(r,q,p) = (0,0,0)$ by \cite[Theorem~4.1]{Coddington1955}. The first three partial derivatives \eqref{partial_p-_0}, \eqref{partial_p-_1r}, and \eqref{partial_p-_1q} evaluated in $(r,q) = (0,0)$, are immediate from \eqref{tangent_space} and the smoothness in $k > 0$ is proved for instance in \cite[p.~165--166]{Palis_Takens} or \cite[\S9.2, Theorem~9.6]{Teschl2012}.

\medskip

We now compute $\rbar{\partial_r^2 p^{-}}{(r,q) = (0,0)}$ in \eqref{partial_p-_2r}. Observe that on $M^{-}$ it holds $p = p^-$, so that
\[
\tfrac{\d p}{\d s} = \partial_r p^{-} \tfrac{\d r}{\d s} + \partial_q p^{-} \tfrac{\d q}{\d s}
\]
and thus using \eqref{ds_f} to substitute derivatives in $s$, we obtain the partial differential equation
\[
\tfrac{3-n}{3} r \partial_r p^{-} + \left(p^{-} - \tfrac{3-n}{3}q\right) \partial_q p^{-} = \tfrac n 3 p^{-} - \tfrac{2}{3 k^3} \tfrac{r^2}{1+r^3} \left(1+rq\right)^{-\frac 1 2},
\]
which is equivalent to \eqref{unstable_eq}. Using the already computed \eqref{partial_p-_0}, \eqref{partial_p-_1r}, and \eqref{partial_p-_1q}, it follows after differentiating \eqref{unstable_eq} in $r$ twice and evaluating at $(r,q) = (0,0)$ that
\[
\left(2 - \tfrac{n}{3-n}\right) \rbar{\left(\partial_r^2 p^{-}\right)}{(r,q) = (0,0)} + \tfrac{3}{3-n} \rbar{\left(\partial_r^2 p^{-}\right)}{(r,q) = (0,0)} = - \tfrac{4}{(3-n) k^3},
\]
leading to $\rbar{\left(\partial_r^2 p^{-}\right)}{(r,q) = (0,0)} = - \tfrac{4}{3 (3-n) k^3}$ as stated in \eqref{partial_p-_2r}.

\medskip

For the proof of \eqref{partial_p-_jrlq} we argue inductively. Taking $\partial_r$ and $\partial_q$ derivatives of \eqref{unstable_eq} we get
\begin{align*}
& \left(r \partial_r - q \partial_q - \tfrac{n+(\ell-j) (3-n)}{3-n}\right) \partial_r^j \partial_q^\ell p^- + \tfrac{3}{3-n} \sum_{0 \le j' \le j} \sum_{0 \le \ell' \le \ell} {j \choose j'} {\ell \choose \ell'} \left(\partial_r^{j-j'} \partial_q^{\ell-\ell'} p^-\right) \left(\partial_r^{j'} \partial_q^{\ell'+1} p^-\right) \\
& \quad = - \tfrac{2 (-1)^{\ell+1}}{(3-n) k^3} \tfrac 1 2 \cdot \tfrac 3 2 \cdot \ldots \cdot \tfrac{2\ell-1}{2} \partial_r^j \left(\tfrac{r^{\ell+2}}{1+r^3} \left(1+rq\right)^{-\frac{2\ell+1}{2}}\right)
\end{align*}
and evaluating at $(r,q) = (0,0)$ leads to
\begin{align}
& \left(n + (\ell-j) (3-n)\right) \rbar{\left(\partial_r^j \partial_q^\ell p^-\right)}{(r,q) = (0,0)} \nonumber\\
& - 3 \sum_{0 \le j' \le j} \sum_{0 \le \ell' \le \ell} {j \choose j'} {\ell \choose \ell'} \rbar{\left(\partial_r^{j-j'} \partial_q^{\ell-\ell'} p^-\right)}{(r,q) = (0,0)} \rbar{\left(\partial_r^{j'} \partial_q^{\ell'+1} p^-\right)}{(r,q) = (0,0)} = 0, \label{partial_induction}
\end{align}
where we suppose $(j,\ell) \in \N_0^2$ with $j \le \ell-2$. If we assume that $\rbar{\left(\partial_r^{j''} \partial_q^{\ell''} p^-\right)}{(r,q) = (0,0)} = 0$ for $(j'',\ell'') \in \N_0^2$ provided
\begin{itemize}
\item $\ell'' \le \ell-1$, or
\item $\ell'' = \ell$ and $j'' \le j-1$,
\end{itemize}
then it follows from \eqref{partial_p-_1q} and \eqref{partial_induction} that
\[
(\ell-j-1) (3-n) \rbar{\left(\partial_r^j \partial_q^\ell p^-\right)}{(0,0,k)} = 0,
\]
which because of $j \le \ell -2$ implies \eqref{partial_p-_jrlq}.
\end{proof}
%

\subsection{The ODE lifted on the unstable manifold}
\subsubsection{Formulation of the ODE}

In what follows, motivated by \eqref{def_r_q_p} and \eqref{partial_p-}, we define
\begin{subequations}\label{def_g_rho}
\begin{equation}\label{def_g}
g := r \rbar{p^-}{q = r^{-1}\mu} - \mu + \tfrac{2}{3 k^3(3-n)} r^3
\end{equation}
and
\begin{equation}\label{def_rho}
\varrho := r^3.
\end{equation}
\end{subequations}
We have the following result:
\begin{corollary}\label{cor:ode_unstable}
Let $n\in(0,3)$. Then the dependent variable $g$ as a function of $(\varrho,\mu,k)$ is analytic in $(\varrho,\mu)$ in a neighborhood of $(\varrho,\mu) = (0,0)$, smooth in $k > 0$, and meets the conditions
\begin{subequations}\label{mu_g_analytic}
\begin{align}\label{bc_g_rho_mu_k}
g &= \partial_\mu g = \partial_\varrho g = 0 && \mbox{at $(\varrho,\mu) = (0,0)$.}
\end{align}
Furthermore, for any in $s \in \R$ twice continuously differentiable $\mu$ solving \eqref{ODE_s} and \eqref{BC_ode_s} it holds for $H > 0$ sufficiently small
\begin{equation}\label{eq_mu_g}
\left(H \tfrac{\d}{\d H} - 1\right) \mu = \rbar{g}{\varrho = H^{3-n}} - \tfrac{2}{3 k^3(3-n)} H^{3-n}.
\end{equation}
\end{subequations}
\end{corollary}
\begin{proof}
Because of \eqref{partial_p-_jrlq} of Lemma~\ref{lem:unstable} and \eqref{def_g}, it holds
\begin{align}\label{g_r_mu_series}
g = \sum_{\substack{j \ge 0, \, \ell \ge 0, \\ j + \ell \ge 1}} \frac{1}{(j+\ell-1)! \ell!} \rbar{\partial_r^{j+\ell-1} \partial_q^\ell p^-}{(r,q) = (0,0)} r^j \mu^\ell - \mu + \frac{2}{3 k^3 (3-n)} r^3,
\end{align}
so that $g$ is analytic in $(r,\mu)$ in a neighborhood of $(r,\mu) = (0,0)$ and smooth in $k > 0$. In view of \eqref{s_h}, \eqref{def_r_q_p}, and \eqref{def_g}, it holds
\[
\left(H \tfrac{\d}{\d H} - 1\right) \mu = \rbar{g}{r = H^{\frac{3-n}{3}}} - \tfrac{2}{3 k^3(3-n)} H^{3-n}.
\]
Because of
\begin{align*}
p^{-} \stackrel{\eqref{def_g}}{=} & r^{-1} g + r^{-1} \mu - \tfrac{2}{3 k^3(3-n)} r^2, \\
r \partial_r p^- \stackrel{\eqref{def_g}}{=} &- r^{-1} g + \partial_r g + r^{-1} \mu \partial_\mu g - \tfrac{4}{3 k^3(3-n)} r^2,\\
\partial_q p^- \stackrel{\eqref{def_g}}{=} &\partial_\mu g + 1, \\
q \partial_q p^- \stackrel{\eqref{def_g}}{=}& r^{-1} \mu\partial_\mu g + r^{-1}\mu,
\end{align*}
on identifying $\mu = r q$, the PDE \eqref{unstable_eq} of Lemma~\ref{lem:unstable} turns into
\begin{align*}
& - r^{-1} g + \partial_r g + r^{-1} \mu \partial_\mu g - \tfrac{4}{3 k^3(3-n)} r^2 - r^{-1} \mu \partial_\mu g - r^{-1} \mu - \tfrac{n}{3-n} r^{-1} g - \tfrac{n}{3-n} r^{-1} \mu + \tfrac{2 n}{3 k^3(3-n)^2} r^2 \\
& + \tfrac{3}{3-n} \left(r^{-1} g + r^{-1} \mu - \tfrac{2}{3 k^3(3-n)} r^2\right) \left(\partial_\mu g + 1\right) \\
& \quad = - \tfrac{2}{(3-n) k^3} \tfrac{r^2}{1+r^3} \left(1+\mu\right)^{-\frac 1 2},
\end{align*}
which simplifies to
\begin{subequations}\label{problem_g_r_mu}
\begin{equation}\label{pde_g_r_mu}
\left((3-n) r \partial_r + 3 \mu \partial_\mu - \tfrac{2}{k^3 (3-n)} r^3 \partial_\mu\right) g + 3 g \partial_\mu g = \tfrac{2}{k^3} r^3 \left(1 - \left(1+r^3\right)^{-1} \left(1+\mu\right)^{-\frac 1 2}\right).
\end{equation}
We obtain with help of \eqref{g_r_mu_series}
\begin{align}
\rbar{g}{(r,\mu) = (0,0)} &= 0, \label{bc_g_r_mu_0} \\
\rbar{\partial_r g}{(r,\mu) = (0,0)} &= \rbar{p^-}{(r,q) = (0,0)} \stackrel{\eqref{partial_p-_0}}{=} 0, \label{bc_g_r_mu_1r} \\
\rbar{\partial_\mu g}{(r,\mu) = (0,0)} &= \rbar{\partial_q p^-}{(r,q) = (0,0)} - 1 \stackrel{\eqref{partial_p-_1q}}{=} 0. \label{bc_g_r_mu_1mu}
\end{align}
\end{subequations}
Writing
\[
g = \sum_{j,\ell = 0}^\infty a_{j,\ell} r^j \mu^\ell \quad \mbox{and} \quad \tfrac{2}{k^3} r^3 \left(1 - \left(1+r^3\right)^{-1} \left(1+\mu\right)^{-\frac 1 2}\right) = \sum_{j,\ell = 0}^\infty c_{j,\ell} r^j \mu^\ell,
\]
where
\begin{subequations}\label{rec_g_r_mu}
\begin{align}\label{coeff_0}
a_{j,\ell} &= c_{j,\ell} = 0 && \mbox{for $(j,\ell) \in \{(0,0),(1,0),(0,1)\}$}
\end{align}
by \eqref{bc_g_r_mu_0}--\eqref{bc_g_r_mu_1mu} and the definition, respectively, we obtain after insertion into \eqref{pde_g_r_mu} the relation
\begin{align}\label{formula_g_r_mu}
a_{j,\ell} &= \frac{c_{j,\ell} + \tfrac{2 (\ell+1)}{k^3 (3-n)} a_{j-3,\ell+1} - 3 \sum_{j'+j''=j} \sum_{\ell'+\ell'' = \ell+1} \ell'' a_{j',\ell'} a_{j'',\ell''}}{(3-n) j + 3 \ell} && \mbox{for $j+\ell \ge 1$},
\end{align}
\end{subequations}
where we let $a_{j-3,\ell+1} = 0$ if $j \le 2$. Note that because of \eqref{coeff_0} it holds $\ell'' a_{j',\ell'} a_{j'',\ell''} = 0$ if $j' + \ell' \ge j+\ell$ or $j''+\ell''\ge j+\ell$. Hence, for $j+\ell = m$ fixed, \eqref{formula_g_r_mu} uniquely determines $a_{j,m-j}$ for $j \in \{0,\ldots,m\}$ inductively starting from $j = 0$. Induction in $m = j + \ell$ using \eqref{rec_g_r_mu} then uniquely determines the coefficients $a_{j,\ell}$ with $(j,\ell) \in \N_0^2$ and thus $g$ in a neighborhood of $(r,\mu) = (0,0)$, where it is analytic in $(r,\mu)$.

\medskip

Using $\varrho = r^3$ and $3 \varrho \partial_\varrho = r \partial_r$, \eqref{pde_g_r_mu}, \eqref{bc_g_r_mu_0}, and \eqref{bc_g_r_mu_1mu} turn into
\begin{subequations}\label{problem_g_rho_mu}
\begin{equation}\label{pde_g_rho_mu}
\left((3-n) \varrho \partial_\varrho + \mu \partial_\mu - \tfrac{2}{3 k^3 (3-n)} \varrho \partial_\mu\right) g + g \partial_\mu g = \tfrac{2}{3 k^3} \varrho \left(1 - \left(1+\varrho\right)^{-1} \left(1+\mu\right)^{-\frac 1 2}\right),
\end{equation}
where
\begin{align}
g &= 0 && \mbox{at $(\varrho,\mu) = (0,0)$ by \eqref{bc_g_r_mu_0}}, \label{bc_g_rho_mu_0} \\
\partial_\mu g &= 0 && \mbox{at $(\varrho,\mu) = (0,0)$ by \eqref{bc_g_r_mu_1mu}}. \label{bc_g_rho_mu_1mu}
\end{align}
Taking a derivative $\partial_\varrho$ of \eqref{pde_g_rho_mu} and using \eqref{bc_g_rho_mu_0} and \eqref{bc_g_rho_mu_1mu}, we infer that
\begin{align}
\partial_\varrho g &= 0 && \mbox{at $(\varrho,\mu) = (0,0)$}.
\end{align}
\end{subequations}
Writing
\[
g = \sum_{j,\ell = 0}^\infty A_{j,\ell} \varrho^j \mu^\ell \quad \mbox{and} \quad \tfrac{2}{3 k^3} \varrho \left(1 - \left(1+\varrho\right)^{-1} \left(1+\mu\right)^{-\frac 1 2}\right) = \sum_{j,\ell = 0}^\infty C_{j,\ell} \varrho^j \mu^\ell,
\]
where
\begin{subequations}\label{rec_g_rho_mu}
\begin{align}\label{a_c_zero}
A_{j,\ell} &= C_{j,\ell} = 0 && \mbox{for $(j,\ell) \in \{(0,0),(1,0),(0,1)\}$}
\end{align}
by \eqref{bc_g_rho_mu_0}--\eqref{bc_g_rho_mu_1mu} and the definition, respectively, we get inserted into \eqref{pde_g_rho_mu} the relation
\begin{align}\label{formula_g_rho_mu}
A_{j,\ell} &= \frac{C_{j,\ell} + \frac{2 (\ell+1)}{3 k^3 (3-n)} A_{j-1,\ell+1} - \sum_{j'+j''=j} \sum_{\ell'+\ell'' = \ell+1} \ell'' A_{j',\ell'} A_{j'',\ell''}}{(3-n) j + \ell} && \mbox{for $j+\ell \ge 1$},
\end{align}
\end{subequations}
where we use the convention $A_{j-1,\ell+1} = 0$ if $j = 0$. Because of \eqref{a_c_zero} we have $\ell'' A_{j',\ell'} A_{j'',\ell''} = 0$ if $j'+\ell' \ge j+\ell$ or $j'' + \ell'' \ge j+\ell$. Thus, for $j + \ell = m$ fixed, \eqref{formula_g_rho_mu} determines $A_{j,m-j}$ with $j \in \{0,\ldots,m\}$ inductively in $j$ starting with $j = 0$. Then all coefficients $A_{j,\ell}$ with $(j,\ell) \in \N_0^2$ are determined by induction in $m = j + \ell$. Hence, problem~\eqref{problem_g_rho_mu} has a solution that is analytic in $(\varrho,\mu)$ in a neighborhood of $(\varrho,\mu) = 0$, thus meeting the boundary conditions \eqref{bc_g_rho_mu_k}. On identifying $\varrho = r^3$, this is in particular a solution to \eqref{problem_g_r_mu} that is analytic in $(r,\mu)$ in a neighborhood of $(r,\mu) = (0,0)$, for which we have proved uniqueness beforehand.
\end{proof}
%

\subsubsection{Uniqueness}

%
\begin{lemma}\label{lem:unstable_unique}
Let $n \in (0,3)$ and $k > 0$. Suppose that $\mu_1$ and $\mu_2$ are continuously differentiable in $H > 0$ and solve
\begin{subequations}\label{ass_unique}
\begin{align}\label{unique_unstable}
\left(H \tfrac{\d}{\d H} - 1\right) \mu_j &= \rbar{g}{(\varrho,\mu) = \left(H^{3-n},\mu_j\right)} - \tfrac{2}{3 k^3 (3-n)} H^{3-n} && \mbox{for $H > 0$ sufficiently small}.
\end{align}
Further suppose that there exists $\delta > 0$ such that
\begin{align}\label{limit_unique}
\lim_{H \searrow 0}  H^{-\delta} \mu_j &= 0 && \mbox{for $j \in \{1,2\}$.}
\end{align}
\end{subequations}
Then it holds
\begin{subequations}\label{limit_diff_unique}
\begin{align}\label{limit_unique_quant}
\mu_j &= \begin{cases} \order\left(H\right) & \text{for } 0 < n < 2, \\
\order\left(- H \ln H\right) & \text{for } n = 2, \\
\order\left(H^{3-n}\right) & \text{for } 2 < n < 3, \end{cases} && \mbox{as $H \downarrow 0$,}
\end{align}
and there exists a constant $\beta \in \R$ such that
\begin{align}\label{diff_muj}
\mu_1 - \mu_2 &= \begin{cases} \beta H \left(1 + \order\left(H\right)\right) & \text{for } 0 < n < 2, \\
\beta H \left(1 + \order\left(- H \ln H\right)\right) & \text{for } n = 2, \\
\beta H \left(1 + \order\left(H^{3-n}\right)\right) & \text{for } 2 < n < 3, \end{cases} && \mbox{as $H \downarrow 0$}.
\end{align}
\end{subequations}
\end{lemma}
\begin{proof}
We have
\begin{eqnarray*}
\left(H \tfrac{\d}{\d H} - 1\right) \mu_j &\stackrel{\eqref{unique_unstable}}{=}& \rbar{g}{(\varrho,\mu) = \left(H^{3-n},\mu_j\right)} - \tfrac{2}{3 k^3 (3-n)} H^{3-n} \\
&=& \rbar{a}{(\varrho,\mu) = \left(H^{3-n},\mu_j\right)} \mu_j - \tfrac{2}{3 k^3 (3-n)} H^{3-n},
\end{eqnarray*}
where by \eqref{bc_g_rho_mu_k} the dependent variable $a$ is a function of $(\varrho,\mu,k)$ being analytic in $(\varrho,\mu)$ and smooth in $k > 0$ such that $a = 0$ at $(\varrho,\mu) = (0,0)$. This implies
\begin{align*}
& H \tfrac{\d}{\d H} \left(H^{-1} \exp\left(\int_H^\eps \rbar{a}{(\varrho,\mu) = \left(\tilde H^{3-n},\rbar{\mu_j}{H = \tilde H}\right)} \tfrac{\d\tilde H}{\tilde H}\right) \mu_j\right) \\
& \quad = - \tfrac{2}{3 k^3 (3-n)} H^{2-n} \exp\left(\int_H^\eps \rbar{a}{(\varrho,\mu) = \left(\tilde H^{3-n},\rbar{\mu_j}{H = \tilde H}\right)} \tfrac{\d\tilde H}{\tilde H}\right)
\end{align*}
for $\eps > 0$ small and thus
\begin{align*}
& H^{-1} \exp\left(\int_H^\eps \rbar{a}{(\varrho,\mu) = \left(\tilde H^{3-n},\rbar{\mu_j}{H = \tilde H}\right)} \tfrac{\d\tilde H}{\tilde H}\right) \mu_j \\
& \quad = \eps^{-1} \rbar{\mu_j}{H = \eps} + \tfrac{2}{3 k^3 (3-n)} \int_H^\eps H_1^{2-n} \exp\left(\int_{H_1}^\eps \rbar{a}{(\varrho,\mu) = \left(H_2^{3-n},\rbar{\mu_j}{H = H_2}\right)} \tfrac{\d H_2}{H_2}\right) \tfrac{\d H_1}{H_1} \\
& \quad = \begin{cases} \order\left(1\right) & \text{for } 0 < n < 2, \\
\order\left(- \ln H\right) & \text{for } n = 2, \\
\order\left(H^{2-n}\right) & \text{for } 2 < n < 3, \end{cases}
\end{align*}
as $H \downarrow 0$. This gives \eqref{limit_unique_quant}.

\medskip

For proving \eqref{diff_muj}, observe that
\begin{eqnarray*}
H^2 \tfrac{\d}{\d H} \left(H^{-1} (\mu_1 - \mu_2)\right) &=& \left(H \tfrac{\d}{\d H} - 1\right) (\mu_1 - \mu_2) \\
&\stackrel{\eqref{unique_unstable}}{=}& \rbar{g}{(\varrho,\mu) = \left(H^{3-n},\mu_1\right)} - \rbar{g}{(\varrho,\mu) = \left(H^{3-n},\mu_2\right)} \\
&=& \rbar{c}{\varrho = H^{3-n}} \left(\mu_1 - \mu_2\right),
\end{eqnarray*}
that is,
\begin{align}\label{ode_diff}
H \tfrac{\d}{\d H} \left(H^{-1} (\mu_1 - \mu_2)\right) = \rbar{c}{\varrho = H^{3-n}} H^{-1} \left(\mu_1 - \mu_2\right) && \mbox{for $H > 0$},
\end{align}
where by \eqref{bc_g_rho_mu_k} the dependent variable $c$ is a function of $(\mu_1,\mu_2,\varrho,k)$ which is analytic in $(\mu_1,\mu_2,\varrho)$ and additionally $c = 0$ at $(\mu_1,\mu_2,\varrho) = (0,0,0)$. Integrating \eqref{ode_diff} from $H = \eps > 0$ yields
\[
H^{-1} \left(\mu_1 - \mu_2\right) = \eps^{-1} \rbar{\left(\mu_1-\mu_2\right)}{H = \eps} \exp\left(-\int_H^\eps \rbar{c}{(\mu_1,\mu_2,\varrho) = \left(\rbar{\mu_1}{H = \tilde H}, \rbar{\mu_2}{H = \tilde H}, \tilde H^{3-n}\right)} \tfrac{\d\tilde H}{\tilde H}\right).
\]
Because of \eqref{limit_unique_quant} the integral $\int_0^\eps \rbar{c}{\varrho = H^{3-n}} \tfrac{\d H}{H}$ is finite, so that the limit
\[
\beta := \lim_{H \searrow 0} H^{-1} \left(\mu_1 - \mu_2\right)
\]
exists. Integrating \eqref{ode_diff} from $H = 0$ then yields
\begin{eqnarray*}
H^{-1} (\mu_1 - \mu_2) &=& \beta \exp\left(\int_0^H \rbar{c}{(\mu_1,\mu_2,\varrho) = \left(\rbar{\mu_1}{H = \tilde H}, \rbar{\mu_2}{H = \tilde H}, \tilde H^{3-n}\right)} \tfrac{\d\tilde H}{\tilde H}\right) \\
&\stackrel{\eqref{limit_unique_quant}}{=}& \begin{cases} \beta \left(1 + \order\left(H\right)\right) & \text{for } 0 < n < 2, \\
\beta \left(1 + \order\left(- H \ln H\right)\right) & \text{for } n = 2, \\
\beta \left(1 + \order\left(H^{3-n}\right)\right) & \text{for } 2 < n < 3, \end{cases}
\end{eqnarray*}
from which \eqref{diff_muj} is immediate.
\end{proof}
%

\subsection{Fixed-point problem}\label{sec:fixed}
In this subsection, we characterize a one-parametric family of solutions to the ordinary initial-value problem (IVP) \eqref{ass_unique} of Lemma~\ref{lem:unstable_unique}. This is split in the non-resonant case in \S\ref{subsect_nonresonant} and the resonant case in \S\ref{subsect_resonant}. Note that resonances have been characterized in \cite[\S4.3]{BGK} in case of the source-type self-similar solution with dynamic nonzero contact angle and that the resonances in the situation at hand are the same. The relevant resonances occur for values $n = 3 - \frac 1 m$, where $m \in \N$.

\medskip

In what follows, suppose that $\mu \in C^0\left([0,\infty)\right) \cap C^1\left((0,\infty)\right)$ meets \eqref{ass_unique}, that is,
\begin{subequations}\label{first_order_mu}
\begin{align}\label{first_order_ode}
\left(H \tfrac{\d}{\d H} - 1\right) \mu &= \rbar{g}{\varrho = H^{3-n}} - \tfrac{2}{3 k^3 (3-n)} H^{3-n} && \mbox{for $H > 0$ sufficiently small}
\end{align}
and
\begin{align}\label{first_order_limit}
\mu &= 0 && \mbox{at } H = 0.
\end{align}
\end{subequations}
In view of \eqref{limit_diff_unique} of Lemma~\ref{lem:unstable_unique} a solution $\mu$ to \eqref{first_order_mu} cannot be expected to be smooth. In what follows we characterize the singularity of $\mu$ in $H = 0$ and the dependence on $k > 0$ explicitly.

\subsubsection{Non-resonant case}\label{subsect_nonresonant}
Consider $n\in(0,3) \setminus \left\{3-\frac{1}{m} \colon m\in\N\right\}$. We unfold the singularity in $H = 0$ by identifying
\begin{equation}\label{identify_nonresonant}
\mu = w+\xi \quad \mbox{provided $\xi = b H$ and $\varrho = H^{3-n}$}
\end{equation}
for a constant $b \in \R$, where $w$ is a function of $(\xi,\varrho,k)$ such that
\begin{subequations}\label{problem_w_nonresonant}
\begin{align}\label{eq4.2}
\left(\xi \partial_\xi + (3-n) \varrho \partial_\varrho-1\right) w &= \rbar{g}{\mu = w+\xi} - \tfrac{2}{3 k^3 (3-n)} \varrho && \mbox{around } (\xi,\varrho)=(0,0)
\end{align}
subject to the boundary conditions
\begin{align}\label{bc_nonresonant}
\left(w,\partial_\xi w\right) &= (0,0) && \mbox{at } (\xi,\varrho) = (0,0).
\end{align}
\end{subequations}
In the following proposition we will construct a solution to \eqref{problem_w_nonresonant} which is analytic in $(\xi,\varrho)$ and smoothly depends on $k>0$. Using \eqref{s_h}, \eqref{def_r_q_p}, \eqref{conv_qcg} of Lemma~\ref{lem:conv_cg}, Corollary~\ref{cor:ode_unstable}, Lemma~\ref{lem:unstable_unique}, and the existence and uniqueness result of \cite{CG} or Theorem~\ref{th:ex_un} in Appendix~\ref{app:ex_un}, it follows that there exists exactly one $b = b_\mathrm{CG} \in \R$ such that $\mu_\mathrm{CG} = w+\xi$ provided $\xi = b_\mathrm{CG} H$ and $\varrho = H^{3-n}$.

\begin{proposition}[non-resonant case]\label{prop:nonresonant}
For $n\in(0,3) \setminus \left\{3-\frac{1}{m} \colon m \in \N\right\}$ problem~\eqref{problem_w_nonresonant} has a solution $w$ which is analytic in $(\xi,\varrho)$ in a neighborhood of $(\xi,\varrho) = (0,0)$ and smooth in $k > 0$.
\end{proposition}
\begin{proof}
The proof of existence of an in $(\xi,\varrho)$ analytic solution to \eqref{problem_w_nonresonant} follows with almost the same reasoning as in \cite[Proposition~4.9]{BGK} using Banach's fixed-point theorem. Since we additionally prove smoothness in $k > 0$, we apply the Banach-space valued version of the implicit-function theorem instead of Banach's fixed-point theorem. Therefore, we rewrite \eqref{problem_w_nonresonant} in the following way: Using a power-series expansion around $(\xi,\varrho)= (0,0)$, it is straight-forward to verify that \eqref{problem_w_nonresonant} is equivalent to
\begin{equation}\label{implicit_nonresonant}
    \cG = 0 \quad \mbox{with} \quad \cG := w - \TT\left[\rbar{g}{\mu = w+\xi} - \tfrac{2}{3 k^3 (3-n)} \varrho\right],
\end{equation}
where the linear operator $\TT$ is defined for in $(\xi,\varrho)$ around $(\xi,\varrho) = (0,0)$ analytic functions $\phi$ with $(\phi,\partial_\xi \phi) = (0,0)$ in $(\xi,\varrho) = (0,0)$ by
\begin{equation*}
\TT \phi := \sum _{(j,\ell)\in \II} \frac{1}{(j+(3-n)\ell-1) j! \ell!} \rbar{\partial_\xi^j \partial_\varrho^\ell \phi}{(\xi,\varrho) = (0,0)} \xi^j \varrho^\ell
\end{equation*}
with $\II := (\mathbb{N}_0)^2\backslash\{(0,0),(1,0)\}$ in view of \eqref{bc_nonresonant}. Note that the choice of $n \notin \left\{3 - \frac 1 m \colon m \in \N\right\}$ and the definition of $\II$ ensure that $j+(3-n)\ell-1 \ne 0$ for all $(j,\ell) \in \II$. In order to construct a solution $w$ to \eqref{implicit_nonresonant}, we use the norm
\begin{equation*}
\vertii{\phi}_\eps := \sum_{(j,\ell) \in \N_0^2} \frac{\eps^{j+2\ell}}{j! \ell!} \verti{\rbar{\partial_\xi^j \partial_\varrho^\ell \phi}{(\xi,\varrho) = (0,0)}}
\end{equation*}
for in $(\xi,\varrho)$ around $(\xi,\varrho) = (0,0)$ analytic $\phi$ with $(\phi,\partial_\xi \phi) = (0,0)$ in $(\xi,\varrho) = (0,0)$, where $\eps > 0$ will be chosen sufficiently small. The corresponding Banach space of all such $\phi$ with $\vertii{\phi}_\eps < \infty$ is denoted by $W_\eps$. From the definition, it is elementary to see that $\vertii{\cdot}_\eps$ is sub-multiplicative, that is, it holds $\vertii{\phi_1 \phi_2}_\eps \le \vertii{\phi_1}_\eps \vertii{\phi_2}_\eps$ for $\phi_1, \phi_2 \in W_\eps$. One further obtains
\begin{align*}
\vertii{\TT \phi}_\eps &= \sum_{(j,\ell)\in \II} \frac{\eps^{j+2\ell}}{j! \ell! \verti{j+(3-n)\ell-1}} \verti{\rbar{\partial_\xi^j \partial_\varrho^\ell \phi}{(\xi,\varrho) = (0,0)}} \le C \sum_{(j,\ell) \in \II} \frac{\eps^{j+2\ell}}{j! \ell!} \verti{\rbar{\partial_\xi^j \partial_\varrho^\ell \phi}{(\xi,\varrho) = (0,0)}} \\
&= C \vertii{\phi}_\eps,
\end{align*}
where $C^{-1} := \min_{(j,\ell) \in \II} \verti{j + (3-n) \ell - 1} > 0$. Hence, $W_\eps \owns \phi \mapsto \cT \phi \in W_\eps$ is a bounded linear operator and thus in particular analytic. For any $w \in W_\eps$, we recognize that by the chain rule $\cG$ is analytic in $w$ with G\^ateaux (and Fr\'echet) derivative
\[
(D_w \cG)\phi = \phi - \cT\left[\rbar{\partial_\mu g}{\mu = w + \xi} \phi\right],
\]
where $\phi \in W_\eps$. With help of Corollary~\ref{cor:ode_unstable} it follows that for $w \in W_\eps$ such that $\vertii{w}_\eps < \delta$ with $\delta > 0$ sufficiently small and $\eps > 0$ sufficiently small there exists $C_1 < \infty$ independent of $\eps$ and $\delta$ such that
\[
\vertii{\cT\left[\rbar{\partial_\mu g}{\mu = w + \xi} \phi\right]}_\eps \le C \vertii{\rbar{\partial_\mu g}{\mu = w + \xi} \phi}_\eps \le C \vertii{\rbar{\partial_\mu g}{\mu = w + \xi}}_\eps \vertii{\phi}_\eps \stackrel{\eqref{bc_g_rho_mu_k}}{\le} C_1 \left(\eps+\delta\right) \vertii{\phi}_\eps.
\]
This implies that for $\delta > 0$ and $\eps > 0$ sufficiently small, $D_w \cG$ is invertible for $w \in W_\eps$ with $\vertii{w}_\eps < \delta$ by the Neumann series, that is, $W_\eps \owns \phi \mapsto D_w \cG \phi \in W_\eps$ is for $w \in W_\eps$ with $\vertii{w}_\eps < \delta$ an isomorphism of Banach spaces.

\medskip

Now, by the chain rule we recognize that $\cG$ has infinitely many mixed G\^ateaux derivatives in directions $w \in W_\eps$ and $k > 0$, so that in particular $W_\eps \times (0,\infty) \owns (w,k) \mapsto \cG \in \R$ is continuously Fr\'echet differentiable. The Banach-space valued implicit-function theorem yields for $\eps > 0$ and $\delta > 0$ small existence of a unique and in $k > 0$ continuously differentiable $w = w_k$ such that \eqref{implicit_nonresonant} holds true. Hence, $w_k$ in particular solves \eqref{problem_w_nonresonant}. Implicitly differentiatiating \eqref{implicit_nonresonant} yields $\partial_k w_k = - \left(\rbar{D_w \cG}{w = w_k}\right)^{-1} \rbar{\partial_k \cG}{w = w_k}$. Now, $\rbar{\partial_k \cG}{w = w_k}$ is continuously differentiable in $k > 0$ and because
\[
\left(\rbar{D_w \cG}{w = w_k}\right)^{-1} = \sum_{j = 0}^\infty \left(\cT\left[\rbar{\partial_\mu g}{\mu = w_k + \xi} \cdot\right]\right)^j,
\]
we see by partially differentiating the above series in $k > 0$ that $\left(\rbar{D_w \cG}{w = w_k}\right)^{-1}$ is continuously differentiable in $k > 0$. Hence, $w_k$ is twice continuously differentiable in $k > 0$ and a bootstrap argument yields smoothness in $k > 0$. As a consequence, we have proved the theorem for $w = w_k$.
\end{proof}
%

\subsubsection{Resonant case}\label{subsect_resonant}
Consider the resonant case $n = 3-\frac{1}{m}$ for an $m \in \N$. We now identify
\begin{equation}\label{identify_resonant}
\mu = w+\xi \quad \mbox{if $\xi = b H$, $\varrho = H^{3-n} = H^{\frac 1 m}$, and $\sigma = H \ln H$}
\end{equation}
for a constant $b \in \R$, where $w$ is a function of $(\xi,\varrho,\sigma,k)$ such that in view of \eqref{first_order_mu} we have
\begin{subequations}\label{problem_w_resonant}
\begin{align}
    \left(m \xi \partial_\xi + \varrho \partial_\varrho + m \left(\sigma+\varrho^m\right) \partial_\sigma - m\right) w &= m \rbar{g}{\mu = w+\xi} - \frac{2m}{3 k^3(3-n)} \varrho && \mbox{around } (\xi,\varrho,\sigma) = (0,0,0), \label{pde_resonant}\\
    \left(w,\partial_\xi w,\partial_\varrho^m w\right) &= (0,0,0) && \mbox{at } (\xi,\varrho,\sigma) = (0,0,0). \label{bc_resonant}
\end{align}
\end{subequations}
The condition $\partial_\varrho^m w = 0$ at $(\xi,\varrho,\sigma) = (0,0,0)$ is necessary in order to exclude non-uniqueness of $w$ under the identification \eqref{identify_resonant}.

\medskip

The following proposition provides an existence result of an in $(\xi,\varrho,\sigma)$ analytic solution to \eqref{problem_w_resonant} which smoothly depends on $k > 0$. With help of \eqref{s_h}, \eqref{def_r_q_p}, \eqref{conv_qcg} of Lemma~\ref{lem:conv_cg}, Lemma~\ref{cor:ode_unstable}, Lemma~\ref{lem:unstable_unique}, and the uniqueness proved in \cite{CG} or Theorem~\ref{th:ex_un} in Appendix~\ref{app:ex_un}, we conclude that there exists exactly one $b = b_\mathrm{CG} \in \R$ such that $\mu_\mathrm{CG} = w+\xi$ provided $\xi = b_\mathrm{CG} H$ and $\varrho = H^{3-n}$.

\begin{proposition}[resonant case]\label{prop:resonant}
Suppose $n = 3-\frac{1}{m}$ for an $m\in\N$. Then \eqref{problem_w_resonant} has a solution $w$ which is analytic in $(\xi,\varrho,\sigma)$ around $(\xi,\varrho,\sigma) = (0,0,0)$ and smooth in $k > 0$.
\end{proposition}
\begin{proof}
As in the proof of Proposition~\ref{prop:nonresonant}, we do not entirely rely on the reasoning in \cite[Proposition~4.10]{BGK}, establishing existence of an analytic solution in an analogous case using Banach's fixed-point theorem, but opt for an application of the implicit-function theorem in order to additionally obtain smoothness in $k > 0$.

\medskip

Therefore, we first define for an in $(\xi,\varrho,\sigma)$ around $(\xi,\varrho,\sigma) = (0,0,0)$ analytic $\phi$ the norm
\begin{equation*}
\vertii{\phi}_\eps := \sum_{(j,\ell,p)\in \N_0^3} \frac{\varepsilon^{j+m\ell+p}}{j!\ell!p!} \verti{\rbar{\partial_\xi^k \partial_\varrho^\ell \partial_\sigma^p \phi}{(\xi,\varrho,\sigma) = (0,0,0)}}.
\end{equation*}
It is easy to see that $\vertii{\cdot}_\eps$ is sub-multiplicative.

\medskip

As a second preliminary step, we consider the linear problem
\begin{subequations}\label{problem_resonant_lin}
\begin{align}
    \left(m \xi \partial_\xi + \varrho \partial_\varrho + m \left(\sigma+\varrho^m\right) \partial_\sigma - m\right) \cT\phi &= \phi && \mbox{around } (\xi,\varrho,\sigma) = (0,0,0), \label{pde_resonant_lin}\\
    \left(\cT\phi,\partial_\xi \cT\phi,\partial_\varrho^m \cT\phi\right) &= (0,0,0) && \mbox{at } (\xi,\varrho,\sigma) = (0,0,0).
\end{align}
\end{subequations}
Choosing $\phi := m \rbar{g}{\mu = w+\xi} - \frac{2m}{3 k^3(3-n)} \varrho$, we recognize that
\begin{subequations}\label{bc_phi_resonant}
\begin{align}
\phi &= m \rbar{g}{\mu = w} \stackrel{\eqref{bc_resonant}} = m \rbar{g}{\mu = 0} \stackrel{\eqref{bc_g_rho_mu_k}}{=} 0 && \mbox{at } (\xi,\varrho,\sigma) = (0,0,0), \\
\partial_\xi \phi &= m \rbar{\partial_\mu g}{\mu = w} (1+\partial_\xi w) \stackrel{\eqref{bc_resonant}} = m \rbar{\partial_\mu g}{\mu = 0}  (1+\partial_\xi w) \stackrel{\eqref{bc_g_rho_mu_k}}{=} 0 && \mbox{at } (\xi,\varrho,\sigma) = (0,0,0), \\
\partial_\sigma \phi &= m \rbar{\partial_\mu g}{\mu = w} \partial_\sigma w \stackrel{\eqref{bc_resonant}} = m \rbar{\partial_\mu g}{\mu = 0}  \partial_\sigma w \stackrel{\eqref{bc_g_rho_mu_k}}{=} 0 && \mbox{at } (\xi,\varrho,\sigma) = (0,0,0).
\end{align}
\end{subequations}
Hence, we may use the power-series expansions
\begin{align*}
\phi &= \sum_{(j,\ell,p) \in \N_0^3} \frac{1}{j! \ell! p!} \rbar{\partial_\xi^j \partial_\varrho^\ell \partial_\sigma^p \phi}{(\xi,\varrho,\sigma) = (0,0,0)} \xi^j \varrho^\ell \sigma^p, \\
\cT \phi &= \sum_{(j,\ell,p) \in \N_0^3} \frac{1}{j! \ell! p!} \rbar{\partial_\xi^j \partial_\varrho^\ell \partial_\sigma^p \cT \phi}{(\xi,\varrho,\sigma) = (0,0,0)} \xi^j \varrho^\ell \sigma^p,
\end{align*}
where in view of \eqref{bc_resonant} and \eqref{bc_phi_resonant} we have
\begin{subequations}\label{resonant_discrete_00}
\begin{align}
\partial_\xi^j \partial_\varrho^\ell \partial_\sigma^p \phi &= 0 && \mbox{in } (\xi,\varrho,\sigma) = (0,0,0) \mbox{ if } (j,\ell,p) \in \N_0^3 \setminus \II, \\
\partial_\xi^j \partial_\varrho^\ell \partial_\sigma^p \cT\phi &= 0 && \mbox{in } (\xi,\varrho,\sigma) = (0,0,0) \mbox{ if } (j,\ell,p) \in \N_0^3 \setminus \JJ,\label{resonant_discrete_t_00}
\end{align}
\end{subequations}
where $\II := \N_0^3 \setminus \left\{(0,0,0),(1,0,0),(0,0,1)\right\}$ and $\JJ := \N_0^3 \setminus \left\{(0,0,0),(1,0,0),(0,m,0)\right\}$. Inserted into \eqref{pde_resonant_lin}, this yields for $(j,\ell,p) \in \N_0^3$ with $\ell < m$,
\begin{subequations}\label{resonant_discrete}
\begin{align}\label{resonant_discrete_l<m}
\left(m j + \ell + m p - m\right) \partial_\xi^j \partial_\varrho^\ell \partial_\sigma^p \cT \phi &= \partial_\xi^j \partial_\varrho^\ell \partial_\sigma^p \phi && \mbox{at } (\xi,\varrho,\sigma) = (0,0,0),
\end{align}
while for $\ell \ge m$ it holds
\begin{align}\label{resonant_discrete_lgem}
\left(m j + \ell + m p - m\right) \partial_\xi^j \partial_\varrho^\ell \partial_\sigma^p \cT \phi + m \frac{\ell!}{(\ell-m)!} \partial_\xi^j \partial_\varrho^{\ell-m} \partial_\sigma^{p+1} \cT \phi &= \partial_\xi^j \partial_\varrho^\ell \partial_\sigma^p \phi && \mbox{at } (\xi,\varrho,\sigma) = (0,0,0).
\end{align}
\end{subequations}
For $(j,\ell,p) \in \left\{(0,0,0),(1,0,0),(0,0,1)\right\}$ equation~\eqref{resonant_discrete_l<m} is fulfilled because of \eqref{resonant_discrete_00}, while for $(j,\ell,p) \in \II$ with $\ell < m$ we get
\begin{subequations}\label{def_t}
\begin{align}\label{def_t_j<m}
\partial_\xi^j \partial_\varrho^\ell \partial_\sigma^p \cT \phi \stackrel{\eqref{resonant_discrete_l<m}}{=} \frac{\partial_\xi^j \partial_\varrho^\ell \partial_\sigma^p \phi}{m j + \ell + m p - m}.
\end{align}
In the case $(j,\ell,p) = (0,m,0)$ it holds
\begin{align}\label{def_t_jlp001}
\partial_\sigma \cT \phi &\stackrel{\eqref{resonant_discrete_lgem}}{=} \partial_\varrho^m \phi && \mbox{at } (\xi,\varrho,\sigma) = (0,0,0)
\end{align}
and for $(j,\ell,p) \in \JJ$ with $\ell \ge m$ we have
\begin{align}\label{def_t_j>m}
\partial_\xi^j \partial_\varrho^\ell \partial_\sigma^p \cT \phi \stackrel{\eqref{resonant_discrete_lgem}}{=} \frac{\partial_\xi^j \partial_\varrho^\ell \partial_\sigma^p \phi}{m j + \ell + m p - m} - m \frac{\ell!}{(\ell-m)!} \frac{\partial_\xi^j \partial_\varrho^{\ell-m} \partial_\sigma^{p+1} \cT \phi}{m j + \ell + m p - m}.
\end{align}
\end{subequations}
Note that equations~\eqref{resonant_discrete_t_00} and \eqref{def_t} uniquely determine $\cT$ by complete induction. Furthermore, in the proof of \cite[Proposition~4.10]{BGK} it is shown how \eqref{def_t} imply that there exists $C < \infty$ independent of $\phi$ and $\eps > 0$ such that $\vertii{\cT \phi}_\eps \le C \vertii{\phi}_\eps$.

\medskip

As in the proof of Proposition~\ref{prop:nonresonant}, we can then reformulate \eqref{problem_w_resonant} as
\begin{equation}\label{implicit_resonant}
    \cG = 0 \quad \mbox{with} \quad \cG := w - \TT\left[m \rbar{g}{\mu = w+\xi} - \tfrac{2 m}{3 k^3 (3-n)} \varrho\right].
\end{equation}
Constructing an in $(\xi,\varrho,\sigma)$ around $(\xi,\varrho,\sigma) = (0,0,0)$ analytic and in $k > 0$ smooth solution to \eqref{implicit_resonant} follows by an application of the Banach-space valued implicit-function theorem. The proof is the same as the one given in Proposition~\ref{prop:nonresonant} as the necessary conditions, the sub-multiplicativity of $\vertii{\cdot}_\eps$, the boundedness of the linear operator $\cT$, and the boundary conditions \eqref{bc_g_rho_mu_k} on $g$ in $(\varrho,\mu) = (0,0)$, remain unchanged.
\end{proof}
%

\section{Proof of the main result\label{sec:proof_main}}
In this section, we prove the main result, Theorem~\ref{main_thm}. This is split into the characterization of two one-parametric solution manifolds $\psi_b$ and $\psi_B$, where $\psi_b$ meets \eqref{main_ode} and \eqref{BC_main_sys_1}, and $\psi_B$ fulfills \eqref{main_ode} and \eqref{BC_main_sys_2} (cf.~\S\ref{sec:solution_manifolds}). These solution manifolds are then matched in three-dimensional phase space $\left(H,\psi,\frac{\d \psi}{\d H}\right)$ using a transversality argument (cf.~\S\ref{sec:transversality}).

\subsection{Solution manifolds at the contact line and in the bulk\label{sec:solution_manifolds}}
The following two propositions characterize the solution manifolds meeting \eqref{main_ode} and \eqref{BC_main_sys_1}, and \eqref{main_ode} and \eqref{BC_main_sys_2}, respectively. The second one, Proposition~\ref{prop:manifold_bulk}, is the same as \cite[Proposition~3.1]{GGO} since the boundary condition \eqref{BC_main_sys_1} at $H = 0$ is immaterial.

\begin{proposition}[solution manifold at the contact line]\label{prop:manifold_contact}
Suppose $n \in (0,3)$. For all $b\in\R$ and $k > 0$ there exists a function $\mu_b$ of $H > 0$ such that
\begin{subequations}\label{manifold_contact}
\begin{align}\label{psib_contact}
\psi_b &= k^2(1+\mu_b) && \mbox{for } H > 0 \mbox{ sufficiently small},
\end{align}
where $\psi_b$ is twice continuously differentiable for $H > 0$ sufficiently small and right-continuous continuous at $H = 0$ solving \eqref{main_ode} for $H > 0$ sufficiently small and \eqref{BC_main_sys_1}, and $\mu_b$ is analytic in $b \in \R$ and smooth in $k > 0$ for $H > 0$ small with
\begin{align}\label{asymptotic_mub_contact}
\partial_b \mu_b &= H \left(1+o(1)\right) && \mbox{as } H \downarrow 0.
\end{align}
\end{subequations}
More precisely, in the non-resonant case $n \in (0,3) \setminus \left\{3 - \frac 1 m \colon m \in \N\right\}$ there exists a function $w$ being analytic in $(\xi,\varrho)$ around $(\xi,\varrho) = (0,0)$ and smooth in $k > 0$ such that $w = 0$ and $\partial_\xi w = 0$ at $(\xi,\varrho) = (0,0)$, and such that
\begin{subequations}\label{psib_contact_precise}
\begin{align}\label{psib_contact_precise_nonresonant}
\mu_b &= b H + \rbar{w}{(\xi,\varrho) = (b H,H^{3-n})} && \mbox{for } H > 0 \mbox{ sufficiently small.}
\end{align}
Likewise, in the resonant case $n = 3 - \frac 1 m$ where $m \in \N$, there exists a function $w$ which is analytic in $(\xi,\varrho,\sigma)$ around $(\xi,\varrho,\sigma) = (0,0,0)$ and smooth in $k > 0$ such that $w = 0$, $\partial_\xi w = 0$, $\partial_\varrho^m w = 0$ at $(\xi,\varrho,\sigma) = (0,0,0)$, and such that
\begin{align}\label{psib_contact_precise_resonant}
\mu_b &= b H + \rbar{w}{(\xi,\varrho,\sigma) = (b H,H^{3-n}, H \ln H)} && \mbox{for } H > 0 \mbox{ sufficiently small.}
\end{align}
\end{subequations}
Furthermore, there exists $b = b_\mathrm{CG} \in \R$ such that $\psi_{b_\mathrm{CG}} = \psi_\mathrm{CG}$, where $\psi_\mathrm{CG}$ is the unique classical solution to \eqref{main_sys} constructed in \cite{CG} or Theorem~\ref{th:ex_un} in Appendix~\ref{app:ex_un}.
\end{proposition}
\begin{proof}
We define $w$ by Proposition~\ref{prop:nonresonant} (non-resonant case) and Proposition~\ref{prop:resonant} (resonant case), respectively. Using that $w$ solves \eqref{problem_w_nonresonant} and \eqref{problem_w_resonant}, respectively, defining $\mu = \mu_b$ through \eqref{identify_nonresonant} and \eqref{identify_resonant}, respectively, we obtain the asymptotics \eqref{psib_contact_precise} and that $\mu_b$ is a solution to problem~\eqref{first_order_mu}. In view of \eqref{def_g_rho}, \eqref{first_order_ode} implies
\begin{align*}
H \partial_H \mu_b &= H^{\frac{3-n}{3}} \rbar{p^-}{q = H^{- \frac{3-n}{3}} \mu_b} && \mbox{for } H > 0,
\end{align*}
so that with $(r,q,p)$ defined as in \eqref{def_r_q_p} and employing \eqref{s_h} we get $p = p^-$. Hence, $(r,q,p)$ lies on the unstable manifold $M^-$ of the stationary point $(r,q,p) = (0,0,0)$ of the dynamical system \eqref{ds_f}. In particular, $\mu$ solves \eqref{ODE_s}, which in view of \eqref{s_h} and defining $\psi$ through \eqref{mu} implies that $\psi$ solves \eqref{main_ode} for $H > 0$ small enough and that \eqref{psib_contact} holds true. The representations \eqref{psib_contact_precise_resonant} as well as \eqref{bc_nonresonant} and \eqref{bc_resonant}, respectively, imply that $\mu_b = 0$ at $H = 0$, which in view of \eqref{psib_contact} shows that \eqref{BC_main_sys_1} is satisfied. Additionally, equations~\eqref{psib_contact_precise_resonant} imply
\[
\partial_b \mu_b = \begin{cases} H + \rbar{\partial_\xi w}{(\xi,\varrho) = \left(bH,H^{3-n}\right)} H & \mbox{for } n \in (0,3) \setminus \left\{\frac 1 m \colon m \in \N\right\}, \\ H + \rbar{\partial_\xi w}{(\xi,\varrho,\sigma) = \left(bH,H^{3-n},H \ln H\right)} H & \mbox{for } n = \frac 1 m \mbox{ with } m \in \N, \end{cases}
\]
which by virtue of \eqref{bc_nonresonant} and \eqref{bc_resonant}, respectively, yields \eqref{asymptotic_mub_contact}.
\end{proof}

We combine this with the following result, which is valid for complete as well as partial wetting:
\begin{proposition}[solution manifold in the bulk, cf.~\cite{GGO}]\label{prop:manifold_bulk}
Suppose $n \in (0,3)$. For all $B>0$ there exists a function $R_B$ of $H > 0$ large enough such that $\psi = \psi_B$ with
\begin{align*}
\psi_B &= \rbar{\psi_\mathrm{CV}}{H \mapsto BH} (1+R_B)&& \mbox{for } H > 0 \mbox{ sufficiently large}
\end{align*}
defines a solution of \eqref{main_ode} and \eqref{BC_main_sys_2}, where $\psi_\mathrm{CV}$ is the unique twice for large $H > 0$ continuously differentiable solution to \eqref{problem_tanner}. Furthermore, it holds
\begin{align*}
R_B &= \order\left(B^{3-n}(\ln(H))^{-1}H^{-(3-n)}\right) && \mbox{as } H \downarrow 0.
\end{align*}
The correction $R_B$ depends, locally in $H$, continuously differentiably on $B > 0$. Additionally, the boundary condition
\begin{align}\label{asymptote_psib_inf}
\partial_H \partial_B \psi_B &= - \tfrac{2}{9B}(\ln(H))^{-\frac{4}{3}}H^{-1}(1+o(1)) && \mbox{as } H\to\infty
\end{align}
holds true. Furthermore, there exists a $B = B_\mathrm{CG}>0$ such that the unique solution $\psi = \psi_\mathrm{CG}$ of \eqref{main_sys} constructed in \cite{CG} or Theorem~\ref{th:ex_un} in Appendix~\ref{app:ex_un} is the same as $\psi_B$.
\end{proposition}
\begin{proof}
See \cite[Proposition~3.1]{GGO} for the statement and \cite[\S4-5]{GGO} for its proof.
\end{proof}
%

\subsection{Matching and transversality\label{sec:transversality}}
This part mainly follows the reasoning in \cite[\S3.3]{GGO} with the difference of deriving continuous differentiability in $k > 0$. Our goal is to study the solution manifolds constructed in Propositions~\ref{prop:manifold_contact} and \ref{prop:manifold_bulk} in three-dimensional phase space $\left(H,\psi,\frac{\d\psi}{\d H}\right)$ which intersect in the unique solution curve $\left(H,\psi_\mathrm{CG},\frac{\d\psi_\mathrm{CG}}{\d H}\right)$.

\begin{lemma}\label{lem:linearized_ode}
Take $n\in(0,3)$, $k > 0$, and let $b = b_\mathrm{CG} \in \R$ and $B = B_\mathrm{CG} > 0$ such that $\psi_b= \psi_B = \psi_\mathrm{CG}$. Then $\psi_b$ and $\psi_B$ are for every $H > 0$ continuously differentiable in $b$ around $b = b_\mathrm{CG}$ and in $B$ around $B = B_\mathrm{CG}$, respectively, and $\eta \in \left\{\rbar{\partial_b \psi_b}{b = b_\mathrm{CG}}, \rbar{\partial_B \psi_B}{B = B_\mathrm{CG}}\right\}$ is twice continuously differentiable in $H > 0$ with
\begin{align}\label{ode_eta}
\tfrac{\d^2\eta}{\d H^2} - \tfrac{1}{3}(H^2+H^{n-1})^{-1} \psi_\mathrm{CG}^{-\frac{3}{2}} \eta &= 0 && \mbox{for } H>0.
\end{align}
\end{lemma}
\begin{proof}
Because of \eqref{psib_contact} and \eqref{psib_contact_precise} of Proposition~\ref{prop:manifold_contact}, the fact that $\psi_{b_\mathrm{CG}} = \psi_\mathrm{CG}$ is a global solution (i.e., a solution of \eqref{main_ode} for all $H > 0$), and continuously differentiable dependence on the data for $H$ taken from any compact subset of $(0,\infty)$ using standard ODE theory, it follows that $\eta = \rbar{\partial_b \psi_b}{b = b_\mathrm{CG}}$ is twice continuously differentiable in $H > 0$ and by differentiating \eqref{main_ode} meets the ordinary differential equation \eqref{ode_eta}. Likewise, using Proposition~\ref{prop:manifold_bulk}, the fact that $\psi_{B_\mathrm{CG}} = \psi_\mathrm{CG}$ is a global solution to \eqref{main_ode}, and standard ODE theory to obtain continuous differentiability on the parameter $B > 0$ for all $H > 0$, taking $\eta = \rbar{\partial_B \psi_B}{B = B_\mathrm{CG}}$, we recognize that $\eta$ is twice continuously differentiable and by differentiating \eqref{main_ode} that \eqref{ode_eta} is satisfied, too.
\end{proof}

We use the following uniqueness result for solutions to \eqref{ode_eta}.
\begin{lemma}[uniqueness of the linearized problem, cf.~\cite{GGO}]\label{lem:unique_linear}
Suppose that $n \in (0,3)$, $k > 0$, and that $\eta$ is twice continuously differentiable in $H > 0$ and right-continous at $H = 0$ such that \eqref{ode_eta},
\begin{subequations}\label{eta_0_inf}
\begin{align}\label{eta_0}
\eta &= 0 && \mbox{at } H=0,
\end{align}
and
\begin{align}\label{eq2_lemma5.3}
\tfrac{\d\eta}{\d H} &\to 0 && \mbox{as } H\to\infty
\end{align}
\end{subequations}
are satisfied. Then $\eta = 0$ for all $H \ge 0$.
\end{lemma}
\begin{proof}
The proof uses the convexity of $\eta^2$ (which easily follows from \eqref{ode_eta}) and is contained in \cite[Lemma~3.3]{GGO}.
\end{proof}

The following corollary implies that the solution manifolds $\left(H,\psi_b,\frac{\d\psi_b}{\d H}\right)$ and $\left(H,\psi_B,\frac{\d\psi_B}{\d H}\right)$ (parametrized by $(b,H)$ and $(B,H)$, and constructed in Propositions~\ref{prop:manifold_contact} and \ref{prop:manifold_bulk}, respectively) intersect transversally in the solution curve $\left(H,\psi_\mathrm{CG},\frac{\d\psi_\mathrm{CG}}{\d H}\right)$ constructed in \cite{CG}.

\begin{corollary}\label{cor:linear_independence}
Suppose $n\in(0,3)$, $k > 0$, and choose $b = b_\mathrm{CG} \in \R$ and $B = B_\mathrm{CG} > 0$ such that $\psi_b= \psi_B = \psi_\mathrm{CG}$. Then the vectors
\[
\left(\rbar{\partial_b \psi_b}{b = b_\mathrm{CG}},\rbar{\partial_H \partial_b \psi_b}{b = b_\mathrm{CG}}\right) \quad \text{and} \quad \left(\rbar{\partial_B \psi_B}{B = B_\mathrm{CG}},\rbar{\partial_H \partial_B \psi_B}{B = B_\mathrm{CG}}\right)
\]
are linearly independent for all $H>0$.
\end{corollary}
\begin{proof}
Because of Propositions~\ref{prop:manifold_contact} and \ref{prop:manifold_bulk}, $b = b_\mathrm{CG} \in \R$ and $B = B_\mathrm{CG} > 0$ such that $\psi_b = \psi_B = \psi_\mathrm{CG}$ exist. By Lemma~\ref{lem:linearized_ode}, $\eta \in \left\{\rbar{\partial_b \psi_b}{b = b_\mathrm{CG}},\rbar{\partial_B \psi_B}{B = B_\mathrm{CG}}\right\}$ is a solution to \eqref{ode_eta} for which by standard theory of ODEs existence and uniqueness of classical solutions for given data $\left(\eta,\frac{\d\eta}{\d H}\right)$ at one $H > 0$ holds true. This implies that $\left(\rbar{\partial_b \psi_b}{b = b_\mathrm{CG}},\rbar{\partial_H \partial_b \psi_b}{b = b_\mathrm{CG}}\right)$ and $\left(\rbar{\partial_B \psi_B}{B = B_\mathrm{CG}},\rbar{\partial_H \partial_B \psi_B}{B = B_\mathrm{CG}}\right)$ are linearly independent for all $H>0$ if they are linearly independent for one $H>0$, which in turn is equivalent to $\rbar{\partial_b \psi_b}{b = b_\mathrm{CG}}$ and $\rbar{\partial_B \psi_B}{B = B_\mathrm{CG}}$ being linearly independent as functions of $H > 0$. The latter will now be proved in the following way: Suppose that
\begin{align}\label{linear_independence}
    \alpha_0 \rbar{\partial_b\psi_b}{b = b_\mathrm{CG}}+\alpha_\infty \rbar{\partial_B\psi_B}{B = B_\mathrm{CG}} = 0 && \mbox{for all } H > 0,
\end{align}
where $\alpha_0,\alpha_\infty \in \R$ are constants. From \eqref{asymptotic_mub_contact} of Proposition~\ref{prop:manifold_contact} we see that $\rbar{\partial_b\psi_b}{b = b_\mathrm{CG}} = 0$ at $H = 0$, $\rbar{\partial_b\psi_b}{b = b_\mathrm{CG}}$ is non-trivial, and from Lemma~\ref{lem:linearized_ode} that $\rbar{\partial_b\psi_b}{b = b_\mathrm{CG}}$ is a solution to the linear ODE \eqref{ode_eta}. By Lemma~\ref{lem:unique_linear} it follows that $\rbar{\partial_H \partial_b\psi_b}{b = b_\mathrm{CG}} \to 0$ as $H \to \infty$ cannot hold. On the other hand, \eqref{asymptote_psib_inf} of Proposition~\ref{prop:manifold_contact} implies $\rbar{\partial_H \partial_B\psi_B}{B = B_\mathrm{CG}} \to 0$ as $H \to \infty$, so that \eqref{linear_independence} yields $\alpha_0 = 0$. Since \eqref{asymptote_psib_inf} of Proposition~\ref{prop:manifold_bulk} also implies that $\rbar{\partial_B\psi_B}{B = B_\mathrm{CG}}$ is nontrivial, we must have $\alpha_\infty = 0$.
\end{proof}

We are now in position to prove our main result.

\begin{proof}[Proof of Theorem~\ref{main_thm}]
By Propositions~\ref{prop:manifold_contact} and \ref{prop:manifold_bulk}, there exist unique $b = b_\mathrm{CG} \in \R$ and $B = B_\mathrm{CG} > 0$ such that $\psi_b = \psi_B= \psi_\mathrm{CG}$. Writing $R_\infty := R_{B_\mathrm{CG}}$ and $v := b \zeta + \rbar{w}{\xi = b \zeta}$, this implies all statements of Theorem~\ref{main_thm} except for the continuously differentiable dependence of $B$ and $R_\infty$ on $k > 0$. In order to prove the latter, define $f := \left(\psi_b-\psi_B, \partial_H \psi_b - \partial_H \psi_B\right)$. Then it holds $f = 0$ for all $H > 0$ if $b = b_\mathrm{CG}$ and $B = B_\mathrm{CG}$. Hence, in particular $f = \partial_H f = 0$ for all $H > 0$ if $b = b_\mathrm{CG}$ and $B = B_\mathrm{CG}$. Corollary~\ref{cor:linear_independence} implies
\begin{align}\label{det_implicit}
    \det \begin{pmatrix} \partial_b f & \partial_B f \\ \partial_b \partial_H f & \partial_B \partial_H f \end{pmatrix} = \det\begin{pmatrix}
    \partial_b \psi_b & -\partial_B \psi_B \\ \partial_b\partial_H \psi_b & -\partial_B \partial_H \psi_B
    \end{pmatrix} \ne 0 && \mbox{for all } H > 0
\end{align}
if $b = b_\mathrm{CG}$ and $B = B_\mathrm{CG}$. Fix a $H > 0$, then $f$ and $\partial_H f$ are functions of $b \in \R$, $B > 0$, and $k > 0$ only and by Propositions~\ref{prop:manifold_contact} and \ref{prop:manifold_bulk} and standard theory of ODEs in the bulk, are smooth in $b \in \R$, continuously differentiable in $B > 0$, and smooth in $k > 0$. Because of \eqref{det_implicit} we infer with help of the implicit-function theorem that $B_\mathrm{CG}$ and $b_\mathrm{CG}$ are continuously differentiable functions of $k > 0$. Since $R_B$ is a continuously differentiable function of $B > 0$, by the chain rule $R_\infty = R_{B_\mathrm{CG}}$ is a continuously differentiable function of $k > 0$.
\end{proof}
%

\appendix

\section{Existence and uniqueness of traveling waves\label{app:ex_un}}
In this appendix, we adapt the existence and uniqueness proof of classical solutions to \eqref{scaled_partial_wetting} in \cite[Theorem~1.1, \S3]{CG} carried out for quadratic mobilities $n = 2$ to prove existence and uniqueness of classical solutions to \eqref{main_sys} for all $n \in (0,3)$. Though there are no significantly new insights, we present the proof for the sake of providing a complete presentation and since in our chosen set of coordinates the proof turns out to be simpler. The proof of uniqueness follows the reasoning of \cite[Lemma~3.3]{GGO}, which is Lemma~\ref{lem:unique_linear} in this note.

\begin{theorem}[cf.~\cite{CG} for $n = 2$]\label{th:ex_un}
Suppose $n \in (0,3)$ and $k > 0$. Then there exists a unique classical solution $\psi = \psi_\mathrm{CG}$ to \eqref{main_sys}, that is, $\psi > 0$ for $H > 0$, and $\psi$ is twice continuously differentiable in $H > 0$ and right-continuous at $H = 0$.
\end{theorem}
\begin{proof}
We first prove uniqueness. Suppose that $\psi_1$ and $\psi_2$ are two classical solutions to \eqref{main_sys}. We set $\phi := \psi_1 - \psi_2$ and have
\begin{align*}
\tfrac{\d^2}{\d H^2} \phi^2 &= 2 \left(\tfrac{\d \phi}{\d H}\right)^2 + 2 \phi \tfrac{\d^2 \phi}{\d H^2} && \mbox{for } H > 0.
\end{align*}
With help of \eqref{main_ode} it follows
\begin{align*}
\phi \tfrac{\d^2 \phi}{\d H^2} &= - \tfrac 2 3 \phi \left(H^2 + H^{n-1}\right)^{-1} \left(\psi_1^{-\frac 1 2} - \psi_2^{- \frac 1 2}\right) \\
&= \tfrac 2 3 \left(H^2 + H^{n-1}\right)^{-1} \psi_1^{-\frac 1 2} \psi_2^{-\frac 1 2} \left(\psi_1^{\frac 1 2} + \psi_2^{\frac 1 2}\right)^{-1} \phi^2 \ge 0 && \mbox{for } H > 0.
\end{align*}
Hence, $\tfrac{\d^2}{\d H^2} \phi^2 \ge 0$ for $H > 0$. Since $\phi = 0$ at $H = 0$ by \eqref{BC_main_sys_1} and $\phi^2 \ge 0$, necessarily $\frac{\d}{\d H} \phi^2 \ge 0$ for $H > 0$ small enough. Because of $\tfrac{\d^2}{\d H^2} \phi^2 \ge 0$ for $H > 0$ we need to have $\frac{\d}{\d H} \phi^2 \ge 0$ for all $H > 0$. This implies with help of \eqref{main_ode}
\begin{align*}
\tfrac{\d}{\d H} \left(\tfrac{\d\phi}{\d H}\right)^2 &= 2 \tfrac{\d\phi}{\d H} \tfrac{\d^2\phi}{\d H^2} = - \tfrac 4 3 \tfrac{\d\phi}{\d H} \left(H^2 + H^{n-1}\right)^{-1} \left(\psi_1^{-\frac 1 2} - \psi_2^{- \frac 1 2}\right) \\
&= \tfrac 2 3 \left(H^2 + H^{n-1}\right)^{-1} \psi_1^{-\frac 1 2} \psi_2^{-\frac 1 2} \left(\psi_1^{\frac 1 2} + \psi_2^{\frac 1 2}\right)^{-1} \tfrac{\d}{\d H} \phi^2 \ge 0 && \mbox{for } H > 0.
\end{align*}
Since $\left(\tfrac{\d\phi}{\d H}\right)^2 \ge 0$ and $\left(\tfrac{\d\phi}{\d H}\right)^2 \to 0$ as $H \to \infty$ by \eqref{BC_main_sys_2}, we obtain $\tfrac{\d\phi}{\d H} = 0$ for all $H > 0$, which together with $\phi = 0$ at $H = 0$ by \eqref{BC_main_sys_1} implies $\phi = \psi_1 - \psi_2 = 0$ for all $H \ge 0$.

\medskip

In order to prove existence, first consider the approximating problems
\begin{subequations}\label{approx_psi}
\begin{align}
\tfrac{\d^2\psi}{\d H^2} + \tfrac 2 3 \left(H^2+H^{n-1}\right)^{-1} \psi^{-\frac 1 2} &= 0 && \mbox{for } \eps < H < \eps^{-1}, \label{approx_psi_ode} \\
\psi &= k^2 && \mbox{at } H = \eps, \label{approx_psi_bc1} \\
\tfrac{\d \psi}{\d H} &= 0 && \mbox{at } H = \eps^{-1},\label{approx_psi_bc2}
\end{align}
\end{subequations}
where $1 > \eps > 0$. Integrating \eqref{approx_psi_ode} twice using the boundary conditions \eqref{approx_psi_bc1} and \eqref{approx_psi_bc2}, we obtain the equivalent fixed-point problem
\begin{align}\label{fixed_psi_eps}
\psi &= \cS[\psi] := k^2 + \tfrac 2 3 \int_\eps^H \int_{H_1}^{\eps^{-1}} \left(H_2^2 + H_2^{n-1}\right)^{-1} \left(\rbar{\psi}{H = H_2}\right)^{-\frac 1 2} \d H_2 \, \d H_1 && \mbox{for } \eps \le H \le \eps^{-1}.
\end{align}
Suppose that $\psi$ is continuous for $\eps \le H \le \eps^{-1}$ with $\psi \ge k^2$. Then we obtain with help of \eqref{fixed_psi_eps} that 
\begin{align}
0 \le \tfrac{\d}{\d H} \cS[\psi] &\le \tfrac{2}{3k} \left(\chi \int_{H}^1 \tilde H^{1-n} \, \d \tilde H + \int_1^{\eps^{-1}} \tilde H^{-2} \, \d \tilde H\right) \nonumber \\
&\le \begin{cases} \tfrac{2}{3k} \left(\chi \tfrac{1-H^{2-n}}{2-n} + 1\right) & \text{if } n \in (0,3) \setminus \{2\} \\ \tfrac{2}{3k} \left(- \chi \ln H +1\right) & \mbox{if } n = 2 \end{cases} \nonumber \\
&\le \begin{cases} \tfrac{2}{3k} \tfrac{3-n}{2-n} & \text{if } 0 < n < 2 \\ \tfrac{2}{3k} \left(-\ln \eps +1\right) & \mbox{if } n = 2 \\
\tfrac{2}{3k} \tfrac{\eps^{2-n}}{n-2} & \text{if } 2 < n < 3 \end{cases} && \mbox{for } \eps \le H \le \eps^{-1}, \label{equi}
\end{align}
where $\chi = 1$ if $0 \le H \le 1$ and $\chi = 0$ else, and
\begin{align}
k^2 \le \cS[\psi] &\le \begin{cases} k^2 + \tfrac{2}{3k} \int_\eps^H \left(\rbar{\chi}{H = \tilde H} \tfrac{1-\tilde H^{2-n}}{2-n} + 1\right) \d \tilde H& \text{if } n \in (0,3) \setminus \{2\} \\ k^2 + \tfrac{2}{3k} \int_\eps^H \left(- \rbar{\chi}{H = \tilde H} \ln \tilde H +1\right) \d \tilde H & \mbox{if } n = 2 \end{cases} \nonumber \\
&\le \begin{cases} k^2 + \tfrac{2}{3 k} \left(\tfrac{\vartheta}{2-n} - \tfrac{\vartheta^{3-n}}{(3-n) (2-n)} + H\right) & \text{if } n \in (0,3) \setminus \{2\} \\ k^2 + \tfrac{2}{3 k} \left(\vartheta - \vartheta \ln \vartheta + H\right) & \mbox{if } n = 2 \end{cases} \nonumber \\
&\le K_\eps := \begin{cases} k^2 + \tfrac{2}{3k} \left(\tfrac{1}{2-n} + \eps^{-1}\right) & \text{for } 0 < n < 2 \\
k^2 + \tfrac{2}{3k} \left(1+\eps^{-1}\right) & \text{for } n = 2 \\
k^2 + \tfrac{2}{3k} \left(\tfrac{1}{(3-n) (n-2)} + \eps^{-1}\right) & \text{for } 2 < n < 3 \end{cases} \label{s_self}
\end{align}
for $\eps \le H \le \eps^{-1}$, where $\vartheta = H$ if $0 \le H \le 1$ and $\vartheta = 1$ if $H > 1$. Denote by $\Psi_\eps$ the set of all on $\eps \le H \le \eps^{-1}$ continuous $\psi$ such that $k^2 \le \psi \le K_\eps$. Then \eqref{s_self} implies that $\cS$ maps $\Psi_\eps$ into itself. By \eqref{equi} the image $\left\{\cS[\psi] \colon \psi \in \Psi_\eps\right\}$ is equi-continous and therefore compact due to the Arzel\`a-Ascoli theorem. Hence, Schauder's fixed-point theorem yields existence of an in $\eps \le H \le \eps^{-1}$ continuous solution $\psi = \psi_\eps$ to \eqref{fixed_psi_eps} which is thus twice continuously differentiable for $\eps \le H \le \eps^{-1}$ and solves \eqref{approx_psi}.

\medskip

As a last step, we pass to the limit $\eps \downarrow 0$ for the approximating solutions $\left(\psi_\eps\right)_{1 > \eps > 0}$, where we continuously extend according to
\begin{align}\label{continue_psi_eps}
\psi_\eps := \begin{cases} \rbar{\psi_\eps}{H = \eps} & \text{for } 0 \le H < \eps. \\
\rbar{\psi_\eps}{H = \eps^{-1}} & \text{for } H > \eps^{-1}.
\end{cases}
\end{align}
Since $\psi_\eps = \cS[\psi_\eps]$, it holds by the first and second line of \eqref{s_self} for any $R > 0$
\begin{align}\nonumber
k^2 \le \psi_\eps &\le \begin{cases} k^2 + \tfrac{2}{3 k} \left(\tfrac{\vartheta}{2-n} - \tfrac{\vartheta^{3-n}}{(3-n) (2-n)} + H\right) & \text{if } n \in (0,3) \setminus \{2\} \\ k^2 + \tfrac{2}{3 k} \left(\vartheta - \vartheta \ln \vartheta + H\right) & \mbox{if } n = 2 \end{cases} \nonumber \\
&\le \begin{cases} k^2 + \tfrac{2}{3k} \left(\tfrac{1}{2-n} + R\right) & \text{for } 0 < n < 2 \\
k^2 + \tfrac{2}{3k} \left(1+R\right) & \text{for } n = 2 \\
k^2 + \tfrac{2}{3k} \left(\tfrac{1}{(3-n) (n-2)} + R\right) & \text{for } 2 < n < 3 \end{cases} && \mbox{for } \eps^{-1} \le H \le R, \label{bound_psi_eps}
\end{align}
which in view of \eqref{continue_psi_eps} implies that $\left(\psi_\eps\right)_{1 > \eps > 0}$ is bounded on $0 \le H \le R$ for any $R > 0$. Furthermore,
\eqref{equi} implies that also $\left(\frac{\d\psi_\eps}{\d H}\right)_{\eps_0 > \eps > 0}$
is almost everywhere bounded on $R_0 \le H \le R_1$ with arbitrary $0 < R_0 < R_1 < \infty$ if $0 < \eps_0 < \min\left\{R_0,R_1^{-1}\right\}$, so that in particular $\left(\psi_\eps\right)_{\eps_0 > \eps>0}$ is equi-continuous on $R_0 \le H \le R_1$. Hence, additionally taking \eqref{approx_psi_ode} and \eqref{bound_psi_eps} into account, also $\left(\frac{\d^2\psi_\eps}{\d H^2}\right)_{\eps_0 > \eps > 0}$
is bounded and equi-continuous on $R_0 \le H \le R_1$ with arbitrary $0 < R_0 < R_1 < \infty$ if $0 < \eps_0 < \min\left\{R_0,R_1^{-1}\right\}$. The Arzel\`a-Ascoli theorem and a diagonal-sequence argument imply that there exists a sub-sequence of $\left(\psi_\eps\right)_{1 > \eps > 0}$, which we do not re-label, and a limiting function $\psi$ depending on $0 < H < \infty$ such that $\left(\frac{\d^j\psi_\eps}{\d H^j}\right)_{1 > \eps > 0}$ converges uniformly to $\frac{\d^j\psi}{\d H^j}$ as $\eps \downarrow 0$ on $R_0 \le H \le R_1$ for all $0 < R_0 < R_1 < \infty$ and $j \in \{0,1,2\}$. In view of \eqref{approx_psi_ode} in particular \eqref{main_ode} is satisfied. Equation~\eqref{continue_psi_eps} and the first line of \eqref{bound_psi_eps} imply that $\psi$ can be continuously extended to $0 \le H < \infty$ with $\psi = k^2$ at $H = 0$, thus verifying \eqref{BC_main_sys_1}. For $H > 0$ and $\eps \ge H^{-1}$ we obtain from \eqref{fixed_psi_eps} that
\begin{align*}
0 \le \tfrac{\d \psi_\eps}{\d H} \le \tfrac{2}{3k} \int_H^{\eps^{-1}} \tilde H^{-2} \, \d\tilde H \le \tfrac{2}{3k} H^{-1},
\end{align*}
which implies that $\frac{\d \psi}{\d H} \to 0$ as $H \to \infty$, thus proving \eqref{BC_main_sys_2}.
\end{proof}
\bibliography{gnann_wisse_tanner_partial_ArXiV_v2}
\bibliographystyle{plain}

\end{document}